\documentclass[10pt]{article}
\usepackage{etex}
\reserveinserts{28}

\usepackage{amsmath,amssymb,epsfig,graphicx}
\usepackage{amsthm}
\usepackage{hyperref}
\usepackage[hmarginratio=1:1, bottom=3cm]{geometry}
\usepackage{tkz-euclide}

\usepackage{epstopdf}
\newcommand{\ms}{\sigma}

\usepackage{floatrow}

\usepackage[utf8]{inputenc} 
\usepackage{color}
\usepackage{enumerate}
\usepackage{setspace}
\usepackage{subfig}
\usepackage{subfloat}
\usepackage{enumitem}
\usepackage{enumerate}
\usepackage{subfig}
\usepackage{enumerate}
\usepackage{enumitem}
\usepackage{mathrsfs}
\usepackage{cancel}

\newtheorem{teo}{Theorem}[section]
\newtheorem{rem}{Remark}[section]
\newtheorem{prop}{Proposition}[section]

\newtheorem{lem}[teo]{Lemma}

\newcommand{\cS}{{\mathcal S}}

\newcommand\R{{\mathbb R}}

\newcommand\N{{\mathbb N}}
\newcommand\Z{{\mathbb Z}}
\newcommand\vare{{\varepsilon}}

\def\harf{\hbox{$\frac{1}{2}$}}

\newcommand{\cE}{{\mathcal E}}

\newcommand{\dx}{\Delta x}

\newcommand{\eps}{\epsilon}
\newcommand{\beqn}{\begin{equation}}
\newcommand{\eeqn}{  \end{equation}}
\newcommand{\beqno}{\begin{equation*}}
\newcommand{\eeqno}{  \end{equation*}}
\newcommand{\be}{\begin{eqnarray}}
\newcommand{\ee}{  \end{eqnarray}}
\newcommand{\beno}{\begin{eqnarray*}}
\newcommand{\eeno}{  \end{eqnarray*}}
\newcommand{\cblue}[1]{{\color{blue} #1}}
\newcommand{\cred} [1]{{\color{red}  #1}}

\newcommand{\conv}{\rightarrow}

\newcommand{\I}{{\mathbb I}}
\def\fud{\hbox{$\frac{1}{2}$}}

\newcommand{\COMMENTED}[1]{}
\newcommand{\MODIF}[1]{\cblue{#1}}
\newcommand{\QUESTION}[1]{\fbox{\cred{#1}}}

\newcommand{\ccS}{{\mathscr S}}
\newcommand{\da}{\Delta a}
\newcommand{\ma}{\alpha}

\numberwithin{equation}{section}

\newtheorem{taggedtheoremx}{}
\newenvironment{taggedtheorem}[1]
 {\renewcommand\thetaggedtheoremx {\underline{\bf #1}}\taggedtheoremx}
 {\endtaggedtheoremx}

\date{}

\begin{document}
\title{High-order filtered schemes for  time-dependent second order HJB equations}
\date{\today}

\author{%
Olivier Bokanowski\footnote{Laboratoire Jacques-Louis Lions, Universit{\'e} Paris-Diderot,
5 Rue Thomas Mann, 75205 Paris, Cedex13,
and Laboratoire UMA, Ensta ParisTech, Palaiseau, 
\texttt{boka\,@\,math.univ-paris-diderot.fr}},
Athena Picarelli\footnote{Mathematical Institute, University of Oxford,
Andrew Wiles Building, Woodstock Rd, Oxford OX2 6GG,
\texttt{\{athena.picarelli,christoph.reisinger\}\,@\,maths.ox.ac.uk}},
and Christoph Reisinger\footnotemark[\value{footnote}]
}

\maketitle

\begin{abstract}In  this paper, we  present and analyse a  class  of ``filtered''  numerical  schemes  
for  second  order Hamilton-Jacobi-Bellman (HJB) equations. 
Our approach follows the ideas introduced in B.D.\ Froese and A.M.\ Oberman, Convergent filtered schemes for the Monge-Amp\`ere partial differential equation, \emph{SIAM J.\ Numer.\ Anal.}, 51(1):423--444, 2013,
and more recently applied by other authors
to stationary or 
time-dependent first order Hamilton-Jacobi equations.
For high order approximation schemes (where ``high''  stands  for  greater  than  one),  
the inevitable loss  of  monotonicity prevents  the  use  of  the  classical theoretical results for convergence to viscosity solutions.
The work introduces a suitable local modification of these schemes by ``filtering'' them with a monotone scheme, such that they
can be proven convergent and still show an overall high order behaviour for smooth enough solutions. 
We give theoretical proofs of these claims and illustrate the behaviour with numerical tests from mathematical finance,
focussing also on the use of backward differencing formulae for constructing the high order schemes.
\end{abstract}

\medskip

{\bf Keywords.}
monotone schemes, high-order schemes, backward difference formulae, viscosity solutions, second order Hamilton-Jacobi-Bellman equations.

\medskip
\section{Introduction}\label{sec:intro}

We consider second order Hamilton-Jacobi-Bellman (HJB) equations in $\R^d$:
\begin{equation}\label{eq:HJB}
\left\{
\begin{array}{ll}
v_t+\sup_{a\in A}\mathcal L^a(t,x,v,D_x v,D^2_x v)=0  & x\in \R^d, t\in (0,T)\\
v(0,x)=v_0(x)  & x\in\R^d,
\end{array}
\right.
\end{equation}
where $\mathcal L^a:[0,T]\times\R^d\times \R\times\R^d\times\cS^{d\times d}\to \R$ {(where $\cS$ is the set of symmetric matrices)} takes the form
\begin{eqnarray}
\label{eq:operator}
\mathcal L^a(t,x,r,p,Q)=\Big\{
   -\frac{1}{2}Tr(\sigma\sigma^T(t,x,a)Q)
   +b(t,x,a)\cdot p
   +f(t,x,a)r
   +\ell(t,x,a)\Big\}
\end{eqnarray}
and  $A\subset\R^{m}$ is a nonempty and compact set. 

Existence and uniqueness of viscosity solutions to \eqref{eq:HJB} are obtained for instance under the classical assumptions
that the initial datum $v_0$ is bounded and Lipschitz continuous, and the coefficients $b,\sigma,f,\ell$ are bounded,
Lipschitz and H\"older continuous respectively in space and time, i.e.\ there is a constant $K$ such that, 
for any $\varphi\in \{b,\sigma,f,\ell\}$,
\be\label{eq:A0}
  \sup_{a\in A}\Bigg\{\sup_{(t,x)\neq (s,y)}\frac{|\varphi(t,x,a)-\varphi(s,y,a)|}{|t-s|^{1/2}+|x-y|} +\sup_{(t,x)}|\varphi(t,x,a)|\Bigg\}\leq K.
\ee
Moreover, the solution is then also Lipschitz continuous in space and locally $1/2$-H\"older continuous in time \cite{CIL92}.
The boundedness assumption can be removed, see also \cite{CIL92}. Moreover, in the paper we consider  the equation on bounded numerical domains so that the boundedness assumption is automatically satisfied.

In this article, we propose approximation schemes for (\ref{eq:HJB}) for which convergence is guaranteed in a general setting, 
and which exhibit high order convergence under sufficient regularity of the solution.
As we will explain in the following,
these are in a sense two conflicting goals, 
and we will meet them by application of a so-called ``filter'', an idea introduced in \cite{FroeOber13}.

The seminal work by Barles and Souganidis \cite{BS91} establishes 
that a consistent and stable scheme converges to the viscosity solution of (\ref{eq:HJB}) 
if it is also monotone. 
That this is not simply a requirement of the proof, but can be crucial in practice, 
is demonstrated, e.g., by \cite{pooleyetal} (and also in Section \ref{sec:UV} here).
 It is shown there empirically that for the uncertain volatility model from \cite{lyons1995uncertain}, 
perhaps the simplest non-trivial second order HJB equation there is, the (consistent and stable but) non-monotone Crank-Nicolson 
scheme fails to converge to the correct viscosity solution in the presence of Lipschitz initial data without higher regularity.
This is in contrast to classical solutions where there is no such monotonicity requirement.
As the setting of solutions above (i.e., Lipschitz in space and $1/2$-H\"older in time) is the natural setting for HJB equations
and often no higher global regularity is observed, a wide literature on monotone schemes has developed.

By Godunov's theorem \cite{godunov1959difference}, in the case of explicit linear schemes for the approximation of the linear advection equation,
the monotonicity property restricts the scheme to be of order at most one. 
Also, in \cite{pij-oos-12}, a similar result is given for the approximation of a diffusion equation and order two.
To the best of our knowledge, for general diffusions in more than one dimension, no monotone schemes of order higher than one are available in the literature.
However, the provable order of convergence for second order HJB equations under the weak assumptions above is significantly less than one.
By a technique pioneered by Krylov based on ``shaking the coefficients'' and mollification to construct smooth sub- and/or super-solutions,
\cite{K97, K00, BJ02, BJ05, BJ07} prove certain fractional convergence orders.
More encouragingly, it was remarked (Augoula and Abgrall \cite{augoula-abgrall-00})
that a weaker ``$\vare$-monotonicity'' property was sufficient for proving convergence towards the viscosity solution.

To make matters worse, in more than one dimension, 
in the presence of general cross-derivative terms even first order consistent  monotone schemes are necessarily ``non-local'', 
by which we mean that the length of the finite difference stencil grows relative to the mesh size as it is refined (see \cite{kocan95} or \cite{reisinger2016non}).
While standard finite difference schemes 
(see, e.g., \cite{kushner2013numerical}) are generally non-monotone unless the diffusion matrix is diagonally dominant (see \cite{crandalllions96}),
schemes which are monotone by construction include semi-Lagrangian schemes \cite{M89,CF95,deb-jak-12} and generalized finite difference schemes \cite{BZ03,BOZ04}.
{In order to utilize} the second-order accuracy of standard finite differences for smooth solutions, 
the authors of \cite{ma-for-2016} use those schemes in regions where the coefficients and controls are such that the scheme is monotone, 
and switches to wide stencils only if monotonicity of the standard scheme, easily checkable by the signs of the {discretization} matrix, fails.

Again contrasting this with the case of more regularity, 
for instance \cite{smears2016discontinuous} prove high order convergence 
of discontinuous finite element approximations under a Cordes condition on the coefficients which guarantees high order Sobolev regularity for smooth enough data.

The simple idea of ``filtered'' schemes is to use a combination of a high order scheme and a low order monotone scheme, 
where the latter is known to converge \emph{a priori} by standard results \cite{BS91}. 
The filter ensures that the low order scheme is used locally where and when the discrepancy to the high order scheme is too large,
thus ensuring at least the same convergence order as the low order scheme (see, e.g., \cite{BJ02, BJ05,   BJ07}), 
but otherwise uses the high order scheme and benefits from its accuracy for smooth solutions.
Such schemes have been proposed, analysed and used in
\cite{FroeOber13} for the Monge-Ampere equation, and in \cite{OberSalv15, BokaFalcSahu15} for first order Hamilton-Jacobi equations
(see also \cite{BokFalFerKalGruZid15} for convergence results  on non-monotone value iteration schemes for first order stationary equations).
We continue 
this program by studying (time-dependent) second order Hamilton-Jacobi-Bellmann equations as they arise from stochastic control problems.

Our results parallel the ones in \cite{BokaFalcSahu15} to prove, in the second order setting, that suitable filtered schemes converge at least of the same order as the underlying monotone scheme if there are solutions (only) in the viscosity sense and exhibit
higher order truncation error for sufficiently smooth solutions.

On one hand, the presence of the diffusion term  implies more regularity of the solution and therefore makes it possible to recover the high order behavior of the scheme (see Sections \ref{sec:test-MV} and \ref{sec:2d}). On the other hand, if compared with the results in \cite{BokaFalcSahu15}, the diffusion has a detrimental effect when the filter is activated to correct some pathological behavior of the high order scheme. In fact, in this case, the loss of accuracy does not remain localized and it diffuses in a neighborhood 
of the region of interest, as clearly shown by our  example in Section \ref{sec:UV}.

 Given the much more restrictive CFL condition on the time step in the second order case, we include implicit time stepping schemes in our analysis. This requires an extension of the arguments in \cite{BokaFalcSahu15} to work with  monotone implicit operators. The monotone schemes covered by the analysis include the most commonly used one-step finite difference and semi-Lagrangian schemes.
The theoretical results of the paper do not make use of any particular assumption on the high order scheme, 
beyond a higher order truncation error.
However, in the numerical examples we focus on backward differentiation formulae (BDF) of second order.
Although non-monotone, 
these schemes show good stability properties and have been recently used for solving
obstacle problems for parabolic differential equations
for American-style options in \cite{Oosterlee03} and \cite{bok-deb-16}.
In this framework, the filter has the role of ensuring the convergence without any important modification of the high order scheme.

The rest of this article is organized as follows. In Section~\ref{sec:def},
we present the general framework for the monotone scheme, the higher order scheme, and the filtered scheme,
and give examples of such schemes.
Section \ref{sec:results} is devoted to convergence results, as well as useful existence results for some implicit schemes.
Numerical examples mostly motivated by problems from mathematical finance are given in Section~\ref{sec:Tests},
and we conclude with some remarks in Section~\ref{sec:concl}.


\medskip

{\bf Acknowledgments.}
This work was partially supported by the program GNCS-INdAM.
We are grateful to Peter Forsyth for the provision of his code for Example 2 (Section \ref{sec:UV}).

\section{Main assumptions and definition of the scheme}\label{sec:def}


In order to simplify the presentation we will focus our analysis on the one-dimensional case:
\be\label{eq:1D}
   & & v_t + \sup_{a\in A} \left(-\fud \sigma^2(t,x,a) v_{xx}  + b(t,x,a) v_x  + f(t,x,a) v + \ell(t,x,a)\right)=0,
\ee
but the main analysis can be extended to higher dimensions.
Let $N\geq 1$ and let us introduce a time step 
$$
   \tau:=T/N.
$$
We denote by $\dx$ the space step. 
A uniform mesh in time and space is defined in one dimension by:
$$
t_n=n\tau,\quad n\in\{0,\ldots, N\}\quad \text{ and }\quad x_i\equiv i \dx\,\quad i\in \mathbb I\subseteq \Z.
$$
We also denote by $\mathcal G_{\Delta x}:=\{x_i: i\in{\mathbb I}\}$ the space grid.
The analysis can be adapted to nonuniform grids (see Section \ref{sec:UV}) and higher dimensions
(see Section \ref{sec:2d}) by interpreting $x_i$ as a general mesh point in a potentially non-uniform or higher-dimensional mesh,
and $\Delta x$ as the maximum mesh size. 

We will denote by $u=(u^n_i)$ the numerical approximation of the solution $v$, so that 
$$
u^n_i\approx v(t_n,x_i)
$$
and furthermore $u^n$ will denote the vector $(u^n_i)_{i\in\mathbb I} $.

We aim to define a high order convergent scheme  (high stands for greater than one) for the approximation of \eqref{eq:HJB}.
However, in order to be in the convergence framework of the theorem of Barles and Souganidis~\cite{BS91}, monotonicity of the scheme is fundamental, restricting the attainable order, as described in the introduction.
Hence, in order to devise our high order scheme, we consider the framework of Froese and Oberman~\cite{FroeOber13} 
using filtered schemes, which is a special form of $\vare$-monotone schemes.
Three main ingredients are needed:
a monotone scheme, a higher order scheme and a filter function. 

Let us consider the numerical approximation given by a monotone convergent (one-step) scheme written,
in abstract form, for $n=0,\dots,N-1$:
\be\label{eq:defUM}
   u^{n+1}_i :=S{_M}(u^n)_i, \quad \forall i\in\I,
\ee
with initialization 
\be\label{eq:u0init}
 u^0_i :=v_0(x_i), \quad \forall i\in \I.
\ee
Although here \eqref{eq:defUM} is written in explicit form, the scheme may be implicitly defined.

Analogously we consider a two-step high order scheme (high order consistent, but possibly neither monotone nor stable), 
for $n\geq 1$:
\be\label{eq:defUH}
   u^{n+1}_i=S_H(u^{n},u^{n-1})_i, \quad \forall i\in\I
\ee
(some particular definition of the scheme might be needed for $u^1$).
As above, the scheme written here in explicit form may also corresponds to an implicit scheme.

A more precise characterization of $S_M$ and $S_H$ will be given below.

We consider the following filter function as introduced in \cite{BokaFalcSahu15, OberSalv15}: 
\be \label{eq:filter}
F(x):=\left\{\begin{array}{ll}
x & \text{if } |x|\leq 1\\
0 & \text{otherwise}.
\end{array}\right.
\ee
Analogous theoretical results may be obtained using different filter functions such that $\|F\|_\infty \leq 1$ and $F(x)=x$ in a neighborhood of $x=0$, as in \cite{FroeOber13}.
The filtered scheme is then defined in the following form, for $n\geq 1$:
\be\label{eq:filterscheme}
  u^{n+1}_i=S{_F}(u^{n},u^{n-1})_i:=S{_M}(u^n)_i+ \vare \tau F\left(\frac{S{_H}(u^{n},u^{n-1})_i-S{_M}(u^n)_i}{\vare \tau}\right), \quad i\in\I,
\ee
where  $\vare=\vare_{\tau,\Delta x} > 0$ and such that 
$$
  \lim_{(\tau,\Delta x)\to 0} \vare_{\tau,\Delta x} =0.
$$
Specific choices of $\vare_{\tau,\Delta x}$ will be made precise later on.

Although the form of the filtered scheme \eqref{eq:filterscheme} is explicit, we emphasize again that 
the computation of $S_M(u^n)$ and of $S_H(u^n,u^{n-1})$ may require the
solution of implicit schemes.

\subsection{The monotone scheme}\label{sec:mono}

For convenience, the monotone scheme \eqref{eq:defUM} shall also be denoted in the following abstract form, for $n=0,\dots,N-1$:
\be\label{eq:schemeM}
   \ccS_{M}(t_{n+1},x_i,u^{n+1}_i,u)_i=0, \quad i\in \I, 
\ee
where $u$ denotes all the components $(u^k_\ell)$.
This formulation may include both explicit and implicit schemes.

For computational purposes it will be necessary to define our scheme on a bounded domain. 
Therefore from now on we consider
$$
 \mathbb I=\{1,\ldots, J\},
$$ 
which also means that the scheme \eqref{eq:schemeM} may take into account some boundary conditions.

We consider a particular family of schemes  $\ccS_{M}$ with the following form:
\begin{eqnarray}
\label{eq:form_of_scheme}
\ccS_M(t_{n+1},x_i,u^{n+1}_i,u)_i \equiv \frac{1}{\tau}\sup_{a\in A}\Big\{ M^{a,n+1} u^{n+1} - G^{a,n}(u^n) \Big\}_i
\end{eqnarray}
where $M^{a,n+1}\in \R^{J\times J}$ and $G^{a,n}(u^n)\in \R^J$.
More explicitly, 
for any $\varphi:[0,T]\times\R\to \R^J$, $\ccS_M$ can be written in the form 
$$
\ccS_M(t_{n+1},x_i,r, \varphi)_i = 
\frac{1}{\tau}\sup_{a\in A}\Big\{ 
     M^{a,n+1}_{ii} r + \sum_{\I\ni j\neq i} M^{a,n+1}_{ij} 
      \varphi(t_{n+1},x_j) - G^{a,n}(\varphi(t_n,.))_i \Big\}.
$$
If the scheme is defined in explicit form, i.e., $u^{n+1}_i= S_M (u^n)_i$, then it suffices to take $M^{a,n+1}=I_J$ 
(the identity matrix in $\R^{J\times J}$) 
and $S_M(u^n)_i := \inf_{a\in A} G^{n,a}(u^n)_i$.

For any $x\in \R^J$ and $A\in \R^{J\times J}$ we denote the usual vector and matrix supremum norm  by
$$
  \| x \|_\infty 
  := \sup_{i \in \I}
  |x_i|\qquad\text{and}\qquad \| A \|_\infty := \sup_{x\neq 0}\frac{\|A x\|_\infty}{\|x\|_\infty} = \sup_{i\in \I}
  \sum_{j\in \I} |A_{i,j}|.
$$
One can observe that in expression \eqref{eq:form_of_scheme} 
only the contribution of the $i$-th line of $M^{a,n+1}$ and $G^{a,n}$ appears. 
Therefore, denoting for any $a\equiv (a_1,\ldots,a_J)\in A^J$
\be\label{eq:extention}
  (M^{a,n})_{i,j} := (M^{a_i,n})_{i,j}\qquad \mbox{and}\qquad (G^{a,n}(u))_{i} := (G^{a_i,n}(u))_{i},
\ee
the scheme can also be written in the equivalent vector form:
\be
  \label{eq:form_of_scheme2}
  \sup_{a\in A^J}\Big\{ M^{a,n+1} u^{n+1} - G^{a,n}(u^n) \Big\}  =  0, \quad \mbox{in $\R^J$}.
\ee

\begin{rem}\label{rem:form_of_scheme} 
The form of the scheme \eqref{eq:form_of_scheme} is natural and will be satisfied by all the schemes considered in this paper.
They are of the ``discretize, then optimise'' type (see \cite{forsyth2007numerical}), 
where we discretize the linear operator in \eqref{eq:operator} by a monotone linear scheme for a fixed control
$a$ and then carry out the optimisation in \eqref{eq:HJB}.
\end{rem}

The following assumptions are considered on $(M^{a,n+1})$ and $(G^{a,n})$:

\medskip
\noindent
{\bf Assumption (A1):}
\begin{itemize}
  \item[$(i)$] For all $a\in A^J$ and $n\geq 1$,
  \be  \label{eq:A1-M-matrix}
      M^{a,n} \ \mbox{is an $M$-matrix}
  \ee
  ($A$ is said to be an $M$-matrix if $A_{ij}\leq 0$, $\forall i\neq j$, and  $A^{-1}\geq 0$ componentwise);
  \item[$(ii)$] For all $n\geq 1$ (and for all $\tau,\dx$),
  there exists a constant $C_n=C_n(\tau,\dx)\geq 0$ such that 
  \be  \label{eq:A1-M-bounded}
      \sup_{a \in A^J} \|M^{a,n}\|_\infty \leq C_n;
  \ee
  \item[$(iii)$] There exists $C>0$  (independent of $n,\tau,\dx$), such that for all $n\geq 1$, 
  \begin{equation}\label{eq:boundNorm}
  \sup_{a\in A^J} \| (M^{a,n})^{-1}\|_\infty\leq 1+C\tau;
  \end{equation}

  \item[$(iv)$] For all $a\in A^J$ and $n\geq 0$, $G^{n,a}$ is monotone increasing, i.e.,
    $\forall \varphi,\psi \in \R^J$:
  \begin{eqnarray}\label{eq:Gmono}
    \varphi\leq \psi  \qquad \Rightarrow \qquad G^{a,n}(\varphi) \le G^{a,n}(\psi)
  \end{eqnarray}
  (where here ``$\leq$'' denotes the componentwise inequality between vectors);

  \item[$(v)$] There exists $C>0$ (independent of $n,\tau,\dx$) such that for all $a \in A^J$, $n\geq 0$, 
  \be\label{eq:Glip}
    \forall \varphi,\psi \in \R^J,  \quad  
    \|G^{a,n}(\varphi) - G^{a,n}(\psi)\|_\infty  \le   (1+C\tau) \|\varphi-\psi\|_\infty,
  \ee
  and
  \be\label{eq:Gbound}
    \|G^{a,n}(0)\|_\infty  \leq   C\tau.
  \ee
\end{itemize}
Existence of solutions of the scheme \eqref{eq:form_of_scheme2} will be shown under assumption (A1)
(see Section~\ref{sec:results}, Proposition~\ref{lem:policy}).

\begin{rem}
Assumption (A1) is more specific than the basic requirements of monotonicity and stability in \cite{BS91}.
However, to show monotonicity for a given scheme for HJB equations, typically (A1)(i) and (iv) are used, and similarly
(A1)(ii), (iii), and (v) for stability. Hence, we consider these as natural assumptions for schemes of the form (\ref{eq:form_of_scheme}).
\end{rem}

Hereafter, we will denote by $C^{p,q}$ the set of functions that are continuously $p$-differentiable 
with repect to the time variable~$t$
and $q$-differentiable with respect to the space variable~$x$, and will also denote $\varphi_{pt}$  
or $\varphi_{qx}$ the corresponding partial derivatives.
For any function $\varphi\in C^{1,2}$, the consistency error of the scheme $\ccS_M$ is defined, for a given $(t,x)$, by:
\be
  & & \cE^\varphi_{\ccS_M}(\tau,\Delta x):= 
  \left|\ccS_{M}(t+\tau,x,\varphi(t+\tau,x),\varphi)-\Big(\varphi_t(t,x)+H(t,x,\varphi,\varphi_x,\varphi_{xx})\Big)\right|
  \nonumber
 \\
  & & \label{eq:EM}
\ee

The following natural consistency property of the scheme 
will be needed.

\medskip

\noindent
{\bf Assumption (A2) [consistency]:}
\begin{itemize}
\item[]
For any $(t,x)$, for any function $\varphi$ sufficiently regular in a neighbourhood of $(t,x)$:
\be\label{eq:truncerror0}
  \lim_{(\tau,\dx,\xi)\to 0} \cE^{\varphi+\xi}_{\ccS_M}(\tau,\dx) = 0.
\ee
\end{itemize}

Notice that the consistency property (A2), introduced for simplicity,
implies the weaker consistency condition of Barles and Souganidis~\cite{BS91} 
in the interior of the domain.

As recalled before, a monotone scheme has a limited order of convergence. 
More precisely, in our setting, we will typically assume to have first order consistency:

\medskip

\noindent
{\bf Assumption (A2') [first-order consistency]:}
\begin{itemize}
\item[]
There exists $q\in\{1,2\}$ such that,  for any function $\varphi \in C^{2,2+q}$:
\be \label{eq:truncerror}
  \big|\mathcal E^\varphi_{\ccS_M}(\tau,\Delta x)\big|\leq C_{M}\ 
   \bigg(\|\varphi_{tt}\|_\infty + \sum_{2\leq k\leq 2+q}\|\varphi_{kx}\|_\infty\bigg) \max(\tau,\dx)
\ee
as $(\tau,\dx)\conv 0$, where $C_M\geq 0$ is a constant independent of $\varphi$, $\tau$ and $\dx$.

We will eventually assume also that the scheme has the same order of convergence in $\tau$ and $\dx$, in order to obtain the bound~\eqref{eq:truncerror}
(see the case of the semi-Lagrangian scheme below).
\end{itemize}

The last inequality reads
$|\mathcal E^\varphi_{\ccS_M}(\tau,\dx)|\leq C^\varphi_{M}\,  \max(\tau,\dx)$ for some constant $C^\varphi_M$ that can be bounded explicitly.

Furthermore, it is easily verified that (A2') $\Rightarrow$ (A2), so that we will focus on (A2').

 
\subsubsection{Two examples of monotone schemes}

We now consider two types of monotone schemes. We focus again on the one-dimensional case~\eqref{eq:1D}.
The first type is the 
Implicit Euler (IE) finite difference scheme, defined as follows:

\begin{taggedtheorem}{Implicit Euler (IE) Scheme}\label{ex:IE} 
For $n\geq 0$ the  scheme is defined by:
\be
  \label{eq:IE}
  & & \frac{u^{n+1}_i - u^{n}_i}{\tau} +\ \sup_{a\in A} \Big\{-\frac{1}{2} \sigma^2(t_{n+1},x_i,a) D^2 u^{n+1}_i
      + b^+(t_{n+1},x_i,a) D^{1,-} u^{n+1}_i    \\
  & & \hspace{2cm}   - \; b^-(t_{n+1},x_i,a) D^{1,+} u^{n+1}_i 
    + f(t_{n+1},x_i,a) u^{n+1}_i + \ell(t_{n+1},x_i,a) \Big\}  \  =  \ 0,
  \nonumber
\ee
where we have denoted 
\be\label{eq:d2v}
 D^2 v_i:=\frac{v_{i-1} - 2 v_i + v_{i+1}}{\dx^2},
\ee
\be \label{eq:d1v}
  D^{1,-} v_i:= \frac{v_{i} - v_{i-1}}{\Delta x}
  \quad \mbox{and} \quad 
  D^{1,+} v_i:= \frac{v_{i+1} - v_{i}}{\Delta x}, 
\ee
and where we have used the decomposition $b=b^+-b^-$ with
\be \label{eq:b-decomposition}
  b^\pm(t,x,a):=\max( \pm b(t,x,a),0).
\ee
\end{taggedtheorem}

The scheme \eqref{eq:IE} can also be written in the equivalent form \eqref{eq:form_of_scheme2} with $M^{a,n+1}$ the tridiagonal matrix such that
\be
 & & M^{a,n+1}_{i,i} :=1+\frac{\tau}{\dx^2}\ms^2(t_{n+1},x_i,a) + \frac{\tau}{\dx}|b(t_{n+1},x_i,a)| +   \tau f(t_{n+1},x_i,a) \\
 & & M^{a,n+1}_{i,i\pm 1} := - \frac{1}{2}\frac{\tau}{\dx^2}\ms^2(t_{n+1},x_i,a) - \frac{\tau}{\dx} b_{\pm}(t_{n+1},x_i,a),
\ee
and with $G^{a,n}$ defined by
\be
  G^{a,n}(u^n)_i := u^n_i  - \tau \ell(t_{n+1},x_i,a).
\ee 

We summarize here some basic results concerning the IE scheme:

\begin{prop}
Assume that $\sigma,b,f,\ell$ are bounded functions. The (IE) scheme satisfies assumptions (A1) and (A2'), with $q=1$.
\end{prop}

\begin{proof} We note that for diagonally dominant matrices, i.e.\ such that $\forall i,$ $|M_{ii}|\geq \delta + \sum_{j\neq i}|M_{ij}|$ 
for some $\delta >0$,
it holds $\|M^{-1}\|_\infty\leq 1/\delta$. From this follows the estimate (A1)($iii$). Other properties are immediate or classical.
\end{proof}

\begin{rem}\label{rem:existenceIE}
Existence and uniqueness results for such implicit schemes will be stated in Section \ref{sec:results},
using monotonicity properties of the matrix $M^{a,n}$.
We can also solve \eqref{eq:IE} efficiently by the policy iteration algorithm
(see \cite{bok-mar-zid-09}).
\end{rem}

We recall that, in multiple dimensions, standard finite difference schemes are in general non-monotone.
In~\cite{BZ03,BOZ04} it is shown how to get monotone schemes 
for second order equations with general diffusion matrices, but also at the cost of a wider stencil as well as limited order of consistency.

\medskip

As an alternative to finite difference schemes,
simple and explicit monotone schemes known as semi-Lagrangian (SL) schemes \cite{M89,CF95,deb-jak-12} can be considered.
They are based on a discrete time approximation 
of the Dynamic Programming Principle satisfied by the exact solution, combined with a spatial grid interpolation.

In the one-dimensional case of \eqref{eq:1D},
this leads to 
the following approximation, for $n\geq 0$:
\be 
  & &  
    u^{n+1}_i = \inf_{a\in A}
        \bigg\{\frac{1}{2} \sum_{\eps=\pm 1} \big[u^n\big]\big( x_i - \tau b(t_n,x_i,a) + \eps \sqrt{\tau} \ms(t_n,x_i,a)\big) 
    \nonumber \\
  & & \hspace{6cm}
       - \tau f(t_n,x_i,a) u^n_i - \tau \ell(t_n,x_i,a)\bigg\},
    \label{eq:SL2}
\ee
where $[\ \cdot\ ]$ stands for a monotone linear interpolation operator on the spatial grid.

A straightforward equivalent of \eqref{eq:SL2},
which shows the similarity with the finite difference scheme (\ref{eq:IE}) and is convenient for computing the truncation error,
 is then:
\begin{taggedtheorem}{Semi-Lagrangian (SL) Scheme}\label{ex:SL} 
For $n\geq 0$ the scheme is defined by:
\be 
 & & \frac{u^{n+1}_i - u^n_i}{\tau}  
  +  \underset{a\in A}\sup \bigg\{
     - \frac{1}{2\tau} \bigg(\sum_{\eps=\pm 1} \big[u^n\big]\big( x_i - \tau b(t_n,x_i,a) + \eps \sqrt{\tau} \ms(t_n,x_i,a)\big) -  2 u^n_i\bigg)  \nonumber\\
 & & \hspace{4.5cm} \phantom{ +  \underset{a\in A}\sup\bigg\{ }
     + f(t_n,x_i,a) u^n_i+\ell(t_n,x_i,a)\bigg\} \ = \ 0.
  \label{eq:SL1}
\ee
\end{taggedtheorem}

As this is an explicit scheme, we can choose $M^{a,n+1}:= I_J$ (the identity matrix in $\R^{J\times J}$) 
and $G^{a,n}(u^n)_i$ defined as the right-hand-side of \eqref{eq:SL2}.
A two-dimensional version will be presented on a numerical example in Section \ref{sec:2d}.

\begin{prop}\label{lem:SL1}
Assume that $\ms,b,f,\ell$ are bounded functions. The (SL) scheme satisfies assumption (A1).
Assumption (A2') is satisfied with $q=2$ and for $\tau$ and $\dx$ of the same order.
\end{prop}
\begin{proof}
For data $\varphi\in C^2$, the interpolation error satisfies 
$\|\varphi - [\varphi]\|_\infty\leq \frac{1}{8} \|\varphi_{xx}\|_\infty \dx^2$, 
and therefore it is easy to see that the consistency error satisfies the following bound:
$$
  \big|\cE^\varphi_{\ccS_M}(\tau,k,\Delta x)\big| \leq C \left(\tau + \frac{\Delta x^2}{\tau}\right).
$$
\COMMENTED{
\footnote{
More precisely, denoting $b=b(t,x,a)$, $\ms=\ms(t,x,a)$ and $\bar x= x - b\tau$, we use the estimate:
\be
\nonumber 
   & & \hspace{-1cm}\sum_{\eps=\pm 1}\frac{1}{2\tau} u(\bar x +\ms\sqrt{\tau}) \\
\nonumber 
   & & =\ \frac{1}{2\tau} \bigg(u(\bar x +\ms\sqrt{\tau}) + u(\bar x - \ms\sqrt{\tau}) - 2u(\bar x)\bigg)   
    + \frac{1}{\tau} (u(\bar x) - u(x)) \\
\nonumber 
   & & =\  \fud \ms^2 u_{xx}(\bar x) + O(\ms^4 \|u_{4x} \|_\infty\tau) -  b u_x(x) + O(b^2 \|u_{xx} \|_\infty \tau)  \\
   & & =\  \fud \ms^2 u_{xx}(x) - b u_x(x) +  
      O\bigg( (\ms^4\|u_{4x} \|_\infty + \ms^2 |b| \| u_{3x}\|_\infty + b^2 \|u_{xx} \|_\infty)\ \tau \bigg).
\ee
}
}
Then for $\dx\equiv \tau$, the desired consistency estimate is obtained.
\end{proof}

Notice that for the exact solution, in general, only Lipschitz spatial regularity holds and precise error estimates are of the order of
$
   O(\tau^{1/4}) + O(\frac{\dx}{\tau}), 
$
where $O(\dx)$ is the interpolation error for Lipschitz regular data, see \cite{deb-jak-12}
(see also \cite{ass-bok-zid-2015} for the case of unbounded data).


We refer to \cite{M89,CF95} for the introduction of SL schemes 
in the context of second order equations and to \cite{deb-jak-12} for an exhaustive discussion and main results.

\subsection{The high order scheme}\label{sec:high}
The high order scheme \eqref{eq:defUH} will be written in the following form, for $n=0,\dots,N-1$:
\be\label{eq:schemeH}
 \ccS{_H}(t_{n+1},x,u^{n+1}_i,u)_i=0, \qquad i\in \mathbb I
\ee
with an initialization of $u^0$ as in \eqref{eq:u0init} and, possibly, a particular definition of $u^1$ to handle the case of two-step schemes (general multi-step schemes can be handled similarly).


We want to make minimal assumptions on the high order scheme to allow flexibility for obtaining the high order
(in particular, we do not assume monotonicity).
As a minimum, we require that the scheme is well-defined, i.e., that (\ref{eq:schemeH}) uniquely determines $u^{n+1}$, such that we can write a general two-step  scheme in explicit form
\beno
u^{n+1} = S_H(u^n,u^{n-1}).
\eeno
We consider the following assumption:
\medskip

\noindent
{\bf Assumption (A3) [high order consistency]:}
\begin{itemize}
\item[]
There exist $k\geq 2$ and $q\in \N$ such that,
for any function $\varphi$ with regularity
$C^{1+k,2+q}$ in the neighborhood of some point $(t,x)$
and such that 
$$ 
  \varphi_t(.,.)+H(.,.,\varphi,\varphi_x,\varphi_{xx})=0
$$
in a neighborhood of $(t,x)$,
one has 
\be\label{eq:consistencyExpl}
 & & \hspace{-2cm} \left|\frac{1}{\tau}\Big(\varphi(t+\tau,x)-S_H(\varphi(t,\cdot),\varphi(t-\tau,\cdot))(x)\Big)
 \right| \nonumber \\
 & &  \leq C_{H,k} \bigg(
       \|\varphi_{(1+k)t}\|_\infty 
       + \sum_{2\leq p\leq 2+q} \|\varphi_{px}\|_\infty 
     \bigg)\ \max(\tau^k, \dx^k)
\ee
for some constant  $C_{H,k}\geq 0$ that is independent of $\varphi$,$\tau$,$\dx$.
\end{itemize}

\begin{rem}
\label{rem:ho}
We point out that for the high order schemes we present here, assumption (A3) is satisfied 
if 

(i) there exist $k\geq 2$, $\ma\in [0,1]$ and $q\in \N$ such that,
\be
  \cE^{\varphi}_{\ccS_H}(\tau,\Delta x)  
   & :=   & 
   \left|\ccS_H(t+\tau,x,\varphi(t+\tau,x),\varphi)
         -\Big(\varphi_t(t+\ma\tau,x)+H(t+\ma\tau,x,\varphi,\varphi_x,\varphi_{xx})\Big)\right| 
   \nonumber
  \\ 
  \label{eq:high-const-CH-bound}
   & \leq & C^\varphi_{H,k}\ \max(\tau^k, \dx^k)
\ee
for some constant  $C^{\varphi}_{H,k}\geq 0$ independent of $\tau$, $\dx$, 

(ii)
the high order scheme is also of the form \eqref{eq:form_of_scheme}, i.e.,
\beno
\ccS_H(t_{n+1},x_i,u^{n+1}_i,u)_i \equiv \frac{1}{\tau}\sup_{a\in A}\Big\{ 
     {\widetilde M}^{a,n+1} u^{n+1} - {\widetilde G}^{a,n}(u^n,u^{n-1}) \Big\}_i
\eeno
(where ${\widetilde M}^{a,n+1}\in \R^{J\times J}$ and ${\widetilde G}^{a,n}(u^n,u^{n-1})\in \R^J$)

(iii) 
   $\sup_{a\in A^J}\|(\widetilde M^{a,n+1})^{-1}\|_{\infty}$ is bounded by a constant independent of $\tau$ and $\dx$.
\end{rem}



We could have asked for consistency error of orders $O(\tau^k + \dx^{k'})$ which are different in time and space. 
However, in this work we will use the same order of consistency in $\tau$ and $\dx$.

\subsubsection{Examples of high order schemes}

Different choices for the high order scheme are possible. 
In this paper we consider mainly second order schemes ($k=2$) based on finite differences
(see Lemma~\ref{lem:high-FD}).

We first focus on a particular second order Backward Difference Formula (BDF2) 
both in time and space as follows (here for the one-dimensional case),
and then discuss
the Crank-Nicolson (CN) scheme, also of second order.

\begin{rem}
Note that (A3) will hold for the BDF2 scheme using $\ma=1$ in Remark \ref{rem:ho}, and for the CN scheme by using $\ma=\fud$,  
which justifies the introduction of this parameter $\ma$. 
\end{rem}

\begin{taggedtheorem}{BDF2 Scheme}
For $n\geq 1$, the scheme is defined by:
\begin{equation}\label{eq:BDF2}
\begin{split}
  \frac{3 u^{n+1}_i - 4u^{n}_i + u^{n-1}_{i}}{2\tau} +\ \sup_{a\in A} \Big\{-\frac{1}{2} \sigma^2(t_{n+1},x_i,a) D^2 u^{n+1}_i
      + b^+(t_{n+1},x_i,a) D^{1,-} u^{n+1}_i & \\
    - \; b ^-(t_{n+1},x_i,a) D^{1,+} u^{n+1}_i 
        + f(t_{n+1},x_i,a) u^{n+1}_i + \ell(t_{n+1},x_i,a) \Big\} & \ = \ 0,
\end{split} 
\end{equation}
where  $D^2 u_i$ corresponds to the usual second order approximation~\eqref{eq:d2v}, 
$b^\pm$ denote the positive (resp. negative) part of $b$ as in \eqref{eq:b-decomposition},
and  a BDF2 approximation (with stencil shifted left or right) is used for the first derivative in space:
\be\label{eq:spaceBDF}
  D^{1,-} v_i:= \frac{3 v_i - 4 v_{i-1} + v_{i-2}}{2\dx}
  \quad \mbox{and} \quad 
  D^{1,+} v_i:= -\bigg(\frac{3 v_i - 4 v_{i+1} + v_{i+2}}{2\dx}\bigg). 
\ee
The first time step $(n=0)$ needs some special treatment and in this case we consider the implicit Euler scheme with the
same spatial BDF2 discretization as in \eqref{eq:BDF2}, i.e., 
the time approximation term $(3 u^{n+1}_i - 4u^{n}_i + u^{n-1}_{i})/(2\tau)$ is replaced by $(u^{n+1}_i - u^{n}_i)/\tau$.
\end{taggedtheorem}

\begin{rem}
Of course, considering the equation in a bounded domain, some modification of the scheme might also be  necessary at the boundary.
\end{rem}

\begin{rem}
The particular treatment of the first order drift term $b(t,x,a) v_x$  is in order to take into account possibly vanishing diffusion terms. 
Indeed, this approximation leads to a second order consistent scheme in both time and space, and appears to be stable 
even when the diffusion term vanishes, $\ms(t,x_i,a)\equiv 0$, as in part of the domain in Example 1, Section \ref{sec:test-MV}. 
The scheme can be solved hereafter in combination with a policy iteration algorithm.
It is also well known that the simple central approximation for $v_x$ (i.e., $(u^{n+1}_{i+1}-u^{n+1}_{i-1})/(2\dx)$)
should be avoided in situations where the diffusion term vanishes.
\end{rem}


To the best of our knowledge, the use of the second order approximations \eqref{eq:spaceBDF} in \eqref{eq:BDF2} is new.
It avoids to switch to a first order backward or forward approximation as in \cite{ma-for-2016}.


\begin{rem}\label{rem:existenceBDF}
We can attempt to solve \eqref{eq:BDF2} by policy iteration. This is done in the numerical section, with no problem encountered.
However, in presence of a non-vanishing drift term $b$, no theoretical results are available at the moment for 
justifying the existence of a solution for this BDF2 scheme.
In particular, one can easily observe that if the finite discretization matrix in front of $u^{n+1}$ is not an M-matrix, 
results such as in \cite{bok-mar-zid-09} do not apply.
\end{rem}


For comparison purposes, and because of its popularity for applications in financial mathematics and engineering, the classical Crank-Nicolson (CN) scheme
will also be tested on Example 2, Section~\ref{sec:UV}.
The precise definition of the scheme used is given in \cite{pooleyetal}.


The following high order consistencies hold, the proof of which is omitted.

\begin{lem}\label{lem:high-FD}
The (BDF2) and the (CN) scheme both satisfy the high-order consistency condition (A3) with $k=2$ and $q = 2$.
\end{lem}


In the one-dimensional case, the previous schemes can be extended to non-uniform grids $(x_i)$. For instance
for $D^2u_i$ one can use the following expression, with $h_i:=x_{i+1}-x_i$:
$$
  D^2u_i = \frac{2}{h_{i-1}+h_i} \bigg(\frac{1}{h_{i-1}} u_{i-1} - \bigg(\frac{1}{h_{i-1}} + \frac{1}{h_{i}}\bigg) u_i + 
      \frac{1}{h_{i}} u_{i+1} \bigg).
$$
This finite difference is generally of first order consistent, and of second order if
$x_i = q(y_i)$ with a uniform grid $y_i$ and a piecewise smooth Lipschitz function $q(\cdot)$,
as we will have in Example 2, Section \ref{sec:UV}. 

%

\section{Main results}\label{sec:results}

We first state the main result on the convergence of the filtered schemes introduced in the previous section.

\begin{teo}
\label{main_teo}
Let assumptions (A1), (A2') be satisfied. 
Let $u$ (resp.\ $u_{M}$) denote the solution of the filtered (resp.\ monotone) scheme. Let $v$ be the viscosity solution of \eqref{eq:1D}.
\begin{enumerate}[label=(\roman*)]
\item
(Convergence of filtered scheme)
If the monotone scheme satisfies the error estimate, for some $\beta>0$,
\be\label{eq:est_mono}
  \max_{0\leq n\leq N} \big\|u^n_{M}-v^n\big\|_\infty\leq 
    C_1 \max(\tau,\dx)^\beta,
\ee
and if in the filtered scheme $\vare$ is chosen such that, for some constant $C\geq 0$,
\be\label{eq:var-est-loworder}
   0<\vare \leq C \max(\tau,\dx)^\beta,
\ee
then the filtered scheme $u^n$ will satisfy the same estimate as for $u_M^n$, i.e.
\beno
  \max_{0\leq n\leq N} \big\|u^n-v^n\big\|_\infty\leq 
     C \max(\tau,\dx)^\beta
\eeno
for some constant $C\geq 0$.
\item (First order convergence in regular cases)
Assume that the viscosity solution has regularity $v\in C^{2,2+q}([0,T],\R)$ (with $q$ as in (A2')).
Assume furthermore that $\vare$ is chosen such that 
$$ 
  0 < \vare \leq c_0 \max(\tau,\dx)
$$
for some constant $c_0\geq 0$. Then a first-order estimate 
holds  for the filtered scheme:
$$
  \max_{0\leq n\leq N} \|u^n-v^n\|_\infty 
     \leq C \max(\tau,\dx)
$$
for some constant $C\geq 0$.
\item (Local high order consistency)
Assume (A3). Let $(t,x)$ be given and assume that 
$v$ is sufficiently regular in a neighborhood of $(t,x)$. 
Assume that 
$$
   \vare:= c_0  \max(\tau,\dx),
$$
with a constant $c_0$ such that
\be\label{eq:CvM}
   C^v_M  := C_M  \bigg(\|v_{tt}\|_\infty  + \sum_{2\leq p\leq 2+q} \|v_{px}\|_\infty \bigg) < c_0
\ee
where the constants $C_M$ and $q$ are as in (A2').
Then, for sufficiently small $t_n-t$, $x_i-x$, $\tau$, $\dx$,
the filtered scheme satisfies the same high order consistency estimate as for the high-order scheme $\ccS_H$.
\end{enumerate}
\end{teo}

\begin{rem}\label{rem:choice_eps}
Motivated by this result, we shall choose  $\vare$ of the form
$$
   \vare = c_0\max(\tau,\Delta x).
$$
If the solution $v$ is sufficiently smooth, we choose $c_0$ such that 
$$
  c_0 > C^v_{M}
$$
(where the constant $C^v_{M}$ is as in the left-hand side of \eqref{eq:CvM})
in order to get the high order consistency $(iii)$. 
Moreover, in general the monotone scheme will satisfy a bound of the form \eqref{eq:est_mono} for some $\beta\in]0,1]$,
and then this choice of $\vare$ ensures that \eqref{eq:var-est-loworder} is satisfied and this gives the convergence 
of the filtered scheme for the case of nonsmooth solutions.
\end{rem}

We give a proof of Theorem \ref{main_teo} at the end of this section, and start by proving some preliminary results on the monotone scheme presented in Section \ref{sec:mono}.

First of all we give an elementary result for solving implicit schemes.
Let $F$ be such that 
\be \label{eq:Fx}
   F(x) = \sup_{a\in A} (M^a x - G^a) \quad \mbox{in $\R^J$.}
\ee
In every timestep of the monotone scheme (\ref{eq:schemeM}) with (\ref{eq:form_of_scheme}),
an equation of the type $F(x)=0$ has to be solved for $x=u^{n+1} \in \mathbb{R}^J$.

Notice that the supremum in (\ref{eq:Fx}) may not be attained in general if the 
 the maps $a\conv M^a$ and $a\conv G^a$
are not continuous.\footnote{We avoided making a continuity assumptions not only because the PDE coefficients may be discontinuous functions of the control, but also because the discretisation may introduce discontinuities if switches between different schemes are utilised to ensure monotonicity (see, e.g., \cite{ma-for-2016}).}

However, considering a maximizing sequence $(M^{a_j},G^{a_j})$ such 
that $\lim_{j\conv \infty} M^{a_j}x- G^{a_j} = F(x)$, since $(M^{a_j},G^{a_j})$ is bounded in a finite dimensional space, 
it is possible to extract a convergent subsequence so that $(M^{a'_j}, G^{a'_j})\conv (M^*,G^*)$ and $M^*x - G^* = F(x)$.
Hence defining 
$$ Q:=\{(M^a,G^a), \ a \in A^J\}, $$
we can write, for any $x\in \R^J$:
\be
\label{eq:sup}
  F(x)=\max_{(M,G)\in \bar Q} (Mx -G),
\ee
where now the supremum is attained in $\bar Q$.

The policy iteration algorithm is then defined as follows: 
\begin{enumerate}
\item
Start from some $x_0\in \R^J$.
\item
Then for $k\geq 0$,
define
\[
(M^k,G^k)\in \arg\max_{(M,G)\in \bar Q} (M x_k-G),
\]
i.e., an element $(M^{k},G^{k})$ in $\bar Q$ such that $M^{k}x_k-G^{k}=F(x_k)$.
\item
Then take 
$x_{k+1}$ the solution of
\[
M^{k}x_{k+1}-G^{k}=0.
\]
\item
Iterate from 2.\ until convergence.
\end{enumerate}

\begin{prop}\label{lem:policy}
Let $A$ be a  non-empty compact set, and let matrices $M^a\in \R^{J\times J}$, vectors $G^a\in \R^J$ be defined 
for $a\in A^J$ as in (\ref{eq:extention}). 
Assume that 
\be\label{eq:lem-monot}
  \forall a\in A^J,\quad   (M^a)^{-1}\geq 0
\ee
and
\be\label{eq:lem-bound}
  \sup_{a\in A} \| M^a\|_\infty \leq C, \quad
  \sup_{a\in A^J} \| (M^a)^{-1}\|_\infty \leq C, \quad
  \sup_{a\in A} \| G^a\|_\infty \leq C
\ee
for some constant $C\geq 0$.

$(i)$ There exists a unique $x\in \R^J$ such that $F(x)=0$, with $F$ from (\ref{eq:sup}).

$(ii)$ For any $x_0 \in \mathbb{R}^J$, the policy iteration algorithm converges to $x$.

$(iii)$ The convergence is superlinear.
\end{prop}

\begin{rem}
The 
result of Proposition~\ref{lem:policy} 
is given in~\cite{bok-mar-zid-09} under a continuity condition of the maps 
$a\conv M^a$ and $a\conv G^a$.
It was more recently shown in~\cite{ma-for-2016} that
the continuity condition is not necessary for convergence.
Furthermore, superlinear convergence 
can be obtained by using the same arguments as in \cite{bok-mar-zid-09}.
\end{rem}

\begin{prop}[Stability and monotonicity]\label{prop:stability}
Let  (A1) be satisfied. 
\\
$(i)$
For any $(\tau,\Delta x)$ there exists a unique solution of \eqref{eq:schemeM} (denoted $(u^n_M)_{n\geq 0}$).
\\
$(ii)$ The scheme $\ccS_M$ is stable in the following sense: 
there exists a constant $C\geq 0$ (independent of $\tau$ and $\Delta x$) such that for any $0\leq n\leq N$
\be\label{eq:boundinf}
  \|u^n_M\|_\infty \leq e^{CT}(\|v_0\|_\infty + CT).
\ee
$(iii)$ 
The scheme is monotone in the sense of Barles and Souganidis~\cite{BS91}, i.e.
\be\label{eq:monot}
  \phi\leq \psi \quad \Rightarrow\quad\ccS_{M}(t,x,r,\phi)\geq \ccS_{M}(t,x,r,\psi).
\ee
\end{prop}

\begin{rem} 
For the monotonicity property \eqref{eq:monot}, only assumptions (A1)$(i)$ and (A1)$(iv)$ are needed.
\end{rem}

\begin{rem}
  As a consequence, under assumptions (A1) and (A2) the scheme $\ccS_M$ (with solution $u_M$) 
satisfies the stability, monotonicity and consistency conditions of the convergence theorem of
Barles and Souganidis~\cite{BS91}. As long as a comparison principle holds for the HJB equation \eqref{eq:HJB}, this proves that the 
monotone scheme $u_M$ converges to the (unique) viscosity solution.
  For numerical purposes we choose 
to deal with a bounded domain, and in principle this would require a precise statement for the viscosity solution of the related
HJB equation with specific boundary conditions.
  However it is not the focus of the paper to go into such a study.
We will rather assume that the monotone scheme deals correctly with the boundary conditions and focus on obtaining a high-order scheme
in the interior of the domain.
\end{rem}

\begin{proof}[Proof of Proposition~\ref{prop:stability}.]

$(i)$ The existence and uniqueness of a solution for \eqref{eq:schemeM} follows from Proposition~\ref{lem:policy}.

$(ii)$ Using assumptions (A1)($iii$) and (A1)($v$)  one has for $0<\tau \leq 1$
\be
   \|u^{n+1}_M\|_\infty 
       & \leq  &  (1+C_1\tau) (C_2\tau +\|u^n_M\|_\infty)  \nonumber \\
       & \leq  &  (1+C \tau) \|u^n_M\|_\infty + C \tau \label{eq:unp1Minf}
\ee
(for some constants $C_1,C_2,C\geq 0$).
By recursion, the bound \eqref{eq:boundinf} is obtained.
 
$(iii)$ Using that 
$$
\ccS_M(t_{n+1},x_i,r, \varphi)_i =  \frac{1}{\tau}\sup_{a\in A}\Big\{ 
     M^{a,n+1}_{ii} r + \sum_{j\neq i} M^{a,n+1}_{ij} \varphi(t_{n+1},x_j) - G^{a,n}(\varphi(t_n,.))_i \Big\},
$$
the non-positivity of the extra diagonal elements of $M^{a,n+1}$ (assumption (A1)($i$)) and the monotonicity of $G$ (assumption (A1)($iv$)),
the monotonicity of $\ccS_M$ follows.
\end{proof}

\begin{prop}\label{prop:mono}
Let assumptions (A1)$(i)$ and (A1)$(iv)$ be satisfied. 
\begin{itemize}
\item[(i)] For any $\phi,\psi\in \R^J$, 
$$
    \phi\leq \psi   \quad \Rightarrow \quad S_{M}(\phi)\leq S_{M}(\psi).
$$
\item[(ii)] Let also assumptions (A1)$(iii)$ and (A1)$(v)$ be satisfied. Then, there exists $C\geq 0$ such that, for any  $\psi,\phi\in\R^J$:
\be\label{eq:S-S}
  \|S_M(\phi)-S_M(\psi)\|_\infty \leq (1+C \tau)\|\phi - \psi\|_\infty.
\ee
\end{itemize}
\end{prop}

\begin{proof}
$(i)$ For explicit schemes the result is straightforward.
Let us now show that this remains true for any implicit finite difference scheme written in the general form (\ref{eq:form_of_scheme}).

Let the optimizing matrix and vector $(M^{*,n+1},G^{*,n})$ be such that 
\be\label{eq:eq1}
  M^{*,n+1} S_M(\psi)-G^{*,n} = 
  \sup_{a\in A^J} \bigg( M^{a,n+1} S_M(\psi) - G^{a,n}(\psi)\bigg) = 0,
\ee
where $(M^{*,n+1},G^{*,n})$ is obtained as a limit of a  convergent subsequence 
 $(M^{a_j,n+1},G^{a_j,n})_j$ that realizes the supremum in \eqref{eq:eq1}, and notice that $(M^{*,n+1})^{-1} \geq 0$ from (A1)($i$).
One also has for any $\phi$ and $a\in A^J$:
$$
  M^{a,n+1} S_M(\phi) - G^{a,n}(\phi) \le 0,
$$
so that, with $a=a_j$ and passing to the limit as $j\conv \infty$:
$$
  M^{*,n+1} S_M(\phi) - G^{*,n}(\phi) \le 0.
$$
Hence 
\be\label{eq:Sphi-Spsi}
  M^{*,n+1} (S_M(\phi) - S_M(\psi)) \le G^{*,n}(\phi) - G^{*,n}(\psi).
\ee
If $\phi\leq \psi$, this implies by A1($iv$)
$$
  M^{*,n+1} (S_M(\phi) - S_M(\psi))\leq 0,
$$
and $(i)$ follows from the monotonicity property of $M^{*,n+1}$.

Let us now prove $(ii)$. Again from \eqref{eq:Sphi-Spsi} together with assumption (A1)$(v)$ one has
$$
  M^{*,n+1} (S_M(\phi) - S_M(\psi)) \le (1+C_1\tau)\|\phi - \psi\|_\infty,
$$
where the quantity in the right-hand side has to be considered as a constant vector. 
Hence, from assumptions (A1)($i$) and (A1)($iii$) we conclude to
$$ 
(S_M(\phi) - S_M(\psi)) \le (1+C_2\tau)(1+C_1\tau)\|\phi - \psi\|_\infty\leq (1+C\tau) \|\phi-\psi\|_\infty.
$$ 
Switching the roles of $\phi$ and $\psi$,
we arrive at the desired estimate~\eqref{eq:S-S}.

\end{proof}

\begin{teo}\label{teo:u-uM}
Let assumption (A1) be satisfied. 
Let $u_M$ and $u$ respectively denote the solution of the monotone and filtered scheme, i.e.:
$$
  u^{n+1}_M = S_M(u^n_M),\qquad \qquad u^{n+1} = S_F(u^n,u^{n-1}).
$$
$(i)$ There exists a constant $C\geq 0$ such that the following estimate holds:
\beno
  \max_{0\leq n\leq N} \|u^n -u^n_M\|_\infty \leq  T e^{CT} \vare.
\eeno
$(ii)$ In particular, 
if the solution of the monotone scheme $u_M$ converges to the viscosity solution of \eqref{eq:HJB}
as $(\tau, \dx) \to 0$, then the solution of the filtered scheme $u$  also.
\end{teo}

\begin{proof}
$(ii)$ is a straighforward consequence of $(i)$ 
recalling that $\vare\to 0$ as $(\tau, \dx) \to 0$. To prove $(i)$,
by the very definition of the filtered scheme \eqref{eq:filterscheme} one has
$$
  u^{n+1}-u_M^{n+1}= S{_M}(u^n)-S{_M}(u^{n}_M)+\vare \tau F(\cdot).
$$
Hence, thanks to Proposition \ref{prop:mono} and the fact that $\|F\|_\infty \leq 1$, it holds:
\beno
  \big\| u^{n+1}-u_{M}^{n+1}\big\|_\infty\leq (1+C\tau) \big\| u^{n}-u^{n}_{M}\big\|_\infty+\vare\tau. 
\eeno
By recursion, using that $1+C\tau\leq e^{C\tau}$ we obtain
\beno
   \| u^{n}_i-u_M^{n}\|_\infty\leq e^{Ct_n} \big(\| u^{0}-u^{0}_{M}\|_\infty+t_n\vare\big).
\eeno
Then using the fact $u^0_M=u^0$, the desired estimate follows.
\end{proof}

Before going on, we state the following preliminary result establishing the first order of convergence of the monotone scheme in the case of smooth solutions. 

\begin{lem}\label{lem:convSM}
Let assumptions (A1), (A2') be satisfied and let $u_M$ denote the solution of the monotone scheme. 
Assume that the viscosity solution $v$ of \eqref{eq:1D} has regularity $C^{2,2+q}([0,T],\R)$ (with $q$ as in (A2')), and let 
$v^n_j=v(t_n,x_j)$ and $v^n=(v^n_j)_{1\leq j\leq J}$.
The following estimate holds:
$$
  \max_{0\leq n\leq N} \|u^n_M-v^n\|_\infty 
     \leq C \max(\tau,\dx),
$$
for some constant $C\geq 0$ independent of $\tau$ and $\dx$.
\end{lem}

\begin{proof}
For a given $n$,  
\beno
   &  & \sup_{a\in A^J}  \frac{1}{\tau} (M^{ a,n+1} u^{n+1}_M - G^{ a,n}(u^n_M))
    \ =\  \ccS_M(t_{n+1},x_i,u^{n+1}_M,u_M)  \ =\ 0.  
\eeno
We can also write
\beno
   \sup_{a\in A^J}  \frac{1}{\tau} (M^{ a,n+1} v^{n+1} - G^{ a,n}(v^n)) = \ccS_M(t_{n+1},x_i,v^{n+1},v).
\eeno
Therefore, by the same argument as in Proposition \ref{prop:mono}($ii$), one obtains for any $n\geq 0$ and for some constants $C\geq 0$:
\beno
   \|u^{n+1}_M - v^{n+1}\|_\infty \leq  (1+C\tau) \|u^{n}_M - v^{n}\|_\infty  + C \tau \|(\ccS_M(t_{n+1},x_i,v^{n+1},v_M))_i \|_\infty.
\eeno
By the consistency assumption (A2') and the fact that $v$ is a classical solution of \eqref{eq:1D}, one has
 $|\cE^v_{\ccS_M}|=|\ccS_M(t_{n+1},x_i,v^{n+1},v)|  \leq c_1 \max(\tau,\dx)$
for some constant $c_1$.
Hence, with $c_2:=C\,c_1$:
\beno
   \|u^{n+1}_M - v^{n+1}\|_\infty \leq  (1+C\tau) \|u^{n}_M - v^{n}\|_\infty  + c_2 \max(\tau,\dx).
\eeno
By recursion, recalling that $u^0_M=v^0$ and using that $1+C\tau\leq e^{C\tau}$, we obtain
\beno
   \|u^n_M-v^n\|_\infty 
        & \leq &  e^{C t_n}(\|u^0_M-v^0\|_\infty + c_2 t_n \max(\tau,\dx)) \\
        & \leq &  c_2 T e^{C T} \max(\tau,\dx). 
\eeno
\end{proof}

We can now give a proof of Theorem \ref{main_teo}.

\begin{proof}[Proof of Theorem \ref{main_teo}]
$(i)$ follows by Theorem~\ref{teo:u-uM} and the fact that 
$$
    \big\|u^n-v^n\big\|_\infty
    \leq \big\|u^n-u^n_M\big\|_\infty + \big\|u^n_M-v^n\big\|_\infty
     \leq Te^{CT}\vare + C_1 \max(\tau,\dx)^\beta.
$$
$(ii)$ is a consequence of Lemma \ref{lem:convSM}. In fact,  as in $(i)$, we have
$$
    \big\|u^n-v^n\big\|_\infty
    \leq \big\|u^n-u^n_M\big\|_\infty + \big\|u^n_M-v^n\big\|_\infty
    \leq Te^{CT}\vare + c_2 T e^{CT} \max(\tau,\dx) 
$$
from which we deduce the desired estimate because of the bound on $\vare$.\\

$(iii)$ The filtered scheme will be equivalent to the high-order scheme ($S_F\equiv S_H$) if 
$$
  \frac{|{S_{H}(v^{n},v^{n-1})_i-S_{M}(v^n)_i}|}{\vare\tau}\leq 1.
$$
Let us first remark that $\frac{1}{\tau}|v^{n+1}_i- S_M(v^n,v^{n-1})_i|$ is of the same order as $\cE_{\ccS_M}^v$. Indeed, being $v$ a classical solution one has 
\beno
\Big|\ccS_M(t_{n+1},x_i,v^{n+1}_i, v)_i\Big| \leq \mathcal E^v_{\ccS_M}(\tau,\Delta x).
\eeno 
Let  $(M^{*,n+1},G^{*,n})$ be the optimal matrix and vector  such that
\beno
   &  & \ccS_M(t_{n+1},x_i,v^{n+1},v)
    \ =\ \frac{1}{\tau} (M^{*,n+1} v^{n+1} - G^{*,n}(v^n))\ =\ \frac{1}{\tau}M^{*,n+1} (v^{n+1} - S_M(v^n)).   
\eeno
Therefore by assumptions (A1)($i$) and (A1)($iii$) one has
\beno
  \Big|\frac{1}{\tau}(v^{n+1} - S_M(v^n))_i\Big| 
   \leq \|(M^{*,n+1} )^{-1}\|_\infty \cE^v_{\ccS_M}(\tau,\Delta x)
   \leq (1+C\tau)\cE^v_{\ccS_M}(\tau,\Delta x) .
\eeno 
Using the high-order consistency assumption (A3) one obtains, for some $k\geq 2$:
\small
\beno
  \frac{|{S_{H}(v^{n},v^{n-1})-S_{M}(v^n)}|}{\eps\tau}
   & \leq & \frac{1}{\eps\tau} \bigg| v^{n+1} - S_M(v^n)\bigg| + \frac{1}{\eps\tau}\bigg|v^{n+1} - S_{H}(v^{n},v^{n-1})\bigg| \\
   & \leq & \frac{1}{\eps} (1+C\tau) |\cE^v_{M}(\tau,\dx)| + \frac{1}{\eps} C \max(\tau,\dx)^k 
\eeno
\normalsize
with $|\cE^v_{M}(\tau,\dx)|\leq C^v_M \max(\tau,\dx)$ and 
$$
  C^v_M = C_M  \bigg(\|v_{tt}\|_\infty  + \sum_{2\leq p\leq 2+q} \|v_{px}\|_\infty \bigg) < c_0.
$$
Therefore $C^v_M/c_0<1$ and for $\max(\tau,\dx)$ sufficiently small, 
\beno
  \frac{|{S_{H}(v^{n},v^{n-1})-S_{M}(v^n)}|}{\eps\tau} 
   & \leq &  \frac{C_M^v}{c_0}(1+C\tau) + C \max(\tau,\dx)^{k-1}\ \leq \ 1.
\eeno
The desired result follows.
\end{proof}

\footnotesize


\normalsize

\section{Numerical tests}\label{sec:Tests}


In this section, we test the filter on three numerical examples which were chosen to illustrate different features.

In the first example, the solution is regular enough for the high order non-monotone scheme to converge and here the filter is not active for appropriately chosen $\varepsilon$ -- 
it ensures convergence without diminishing the accuracy.
In the second example, a non-smoothness in the initial data causes a non-monotone time-stepping scheme to converge to a wrong value, while the filtered scheme corrects the behaviour at the cost of lower order convergence.
Finally, the last example is of a degenerate but smooth two-dimensional diffusion problem, where a local high-order finite difference scheme is necessarily non-monotone, and again the filter can be used to ensure convergence without sacrificing the observed high order. 

\subsection{Example 1: Mean-variance asset allocation problem}\label{sec:test-MV}

We study the mean-variance asset allocation problem 
(see \cite{forsythwang10, ling00,zhouli00}) formulated as follows on $\Omega:=(x_{min},x_{max})\subset\R$:
\be
  \label{meanvar1d}
  & & v_t + \sup_{a \in A} 
    \bigg(-\harf (\sigma  a x)^2 v_{xx} - (c + x (r + a \sigma \xi))v_x\bigg) = 0, \quad t\in(0,T), \ x\in \Omega
  \\
  & & v(0,x)  =  \left(x-\frac{\gamma}{2}\right)^2, \quad x\in \Omega.
  \label{meanvarbc}
\ee
Here, $x$ is the wealth of an investor who controls the fraction $a$ to invest in a risky asset 
in order to minimize the portfolio variance under a return target. 
The above equation assumes a Black-Scholes model with volatility $\sigma$
and Sharpe ratio $\xi$, $r$ the rate on a risk-free bond and $c$ the contribution rate, $\gamma$ a measure of the risk aversion.
We use the parameters from \cite{forsythwang10}, see Table \ref{tab:parammeanvar}.

\begin{table}[!hbtp]
\centering
\begin{tabular}{|c|c|c|c|c|c|c|}
\hline 
$r$ & $\sigma$ & $\xi$ & $c$  & $T$ & $\gamma$ \\
\hline
$0.03$ & $0.15$ & $0.33$ & $0.1$  & $20$ & $14.47$ \\
\hline
\end{tabular}
\caption{Parameters used in numerical experiments for mean-variance problem.}
\label{tab:parammeanvar}
\end{table}

If bankruptcy ($x<0$) is not allowed, the PDE (\ref{meanvar1d}--\ref{meanvarbc}) holds on $\Omega=(0,\infty)$.

As in \cite{forsythwang10}, the control is assumed to take values in the bounded set 
$$
  A := [0, 1.5].
$$

The equation at $x=0$ simplifies to
\begin{eqnarray}
\label{bdypde}
  v_{t}(t,0) - c v_x(t,0) = 0
\end{eqnarray}
(see \cite{forsythwang10} for a discussion), which 
is a pure transport equation and, since $c>0$, there is an influx boundary so that {\em no boundary condition} is needed at $x=0$.

For numerical purposes, we truncate the computational domain to
\beno
  \Omega =(0, 5).
\eeno
As in \cite{reisinger2016piecewise}, at the boundary $x_{max}=5$ we consider the Dirichlet boundary condition obtained with the solution of the 
equation associated with the asymptotic optimal control $a\equiv 0$ (obtained for large values of $x$), leading us to solve $\bar v_t - (c + rx) \bar v_x=0$.
By using the method of characteristics, we obtain the following boundary value:
\beno
  \bar v(t,x_{max}):= v_0\bigg(\frac{e^t(c+ r x_{max})-c}{r}\bigg).
\eeno

Let $(x_j)$ be a uniform mesh on $[x_{min},x_{max}]$: 
$x_j = x_{min}+ j \dx$, $j=0,\dots,J$, with
$$
  \dx=\frac{x_{max}-x_{min}}{J},
$$
and $t_n=n\tau$, $\tau=\frac{T}{N}$.

We consider a filtered scheme using the BDF2 scheme~\eqref{eq:BDF2} as the high order scheme, and 
the implicit Euler scheme~\eqref{eq:IE} as the monotone scheme. 
We choose here 
$\vare :=c_0\max(\tau,\Delta x)$, with $N=J$ and 
therefore $\tau$ and $\dx$ are of the same order (in particular $\tau = 4\Delta x$). 

Figure \ref{fig:nonsmoothsols} shows the approximate shape of the value function and of the optimal control computed with the filtered scheme.
The optimal control is at the upper bound for small $x$, then decreases linearly to zero, the lower bound, and then stays constant at~0.
A loss of regularity can be observed for $x\sim 2.5$ corresponding to a switching point of the control,
as shown  in Figure \ref{fig:nonsmoothsols} (right) where the second order derivative in space has a jump.

The global errors for different norms, obtained with the choice $\vare = c_0 \tau$ and  $c_0 = 5$,  are given 
in Table \ref{tab:bdf_high}. 
Table~\ref{tab:bdf_high10_local} also gives the local errors computed away from the singularity located 
at $x\sim 2.5$ (i.e.\ on $[0,\,5]\backslash[2.3,\,2.7]$).
For the computation of the errors we have used as reference an accurate numerical solution obtained with $N=J=163840$ and with 
the high order scheme.

The scheme shows a clear second order of convergence for the $L^1$- and $L^2$-norms.
The order in the $L^\infty$-norm deteriorates around 
$x\sim 2.5$. 
Away from the singularity, the scheme is also second order in the $L^\infty$-norm.

Here the use of the filter secures (via Theorem \ref{main_teo}) the convergence of the overall approximation towards the viscosity solution, while the BDF2 approximation leads to practically observed second order behavior in both time and space.

We use policy iteration to solve the discretised nonlinear systems, and although convergence is not guaranteed in the case
of the non-monotone scheme (see Remark \ref{rem:existenceBDF}), we did not encounter any problems.


\begin{figure}
\centering
\includegraphics[width=0.49\columnwidth]{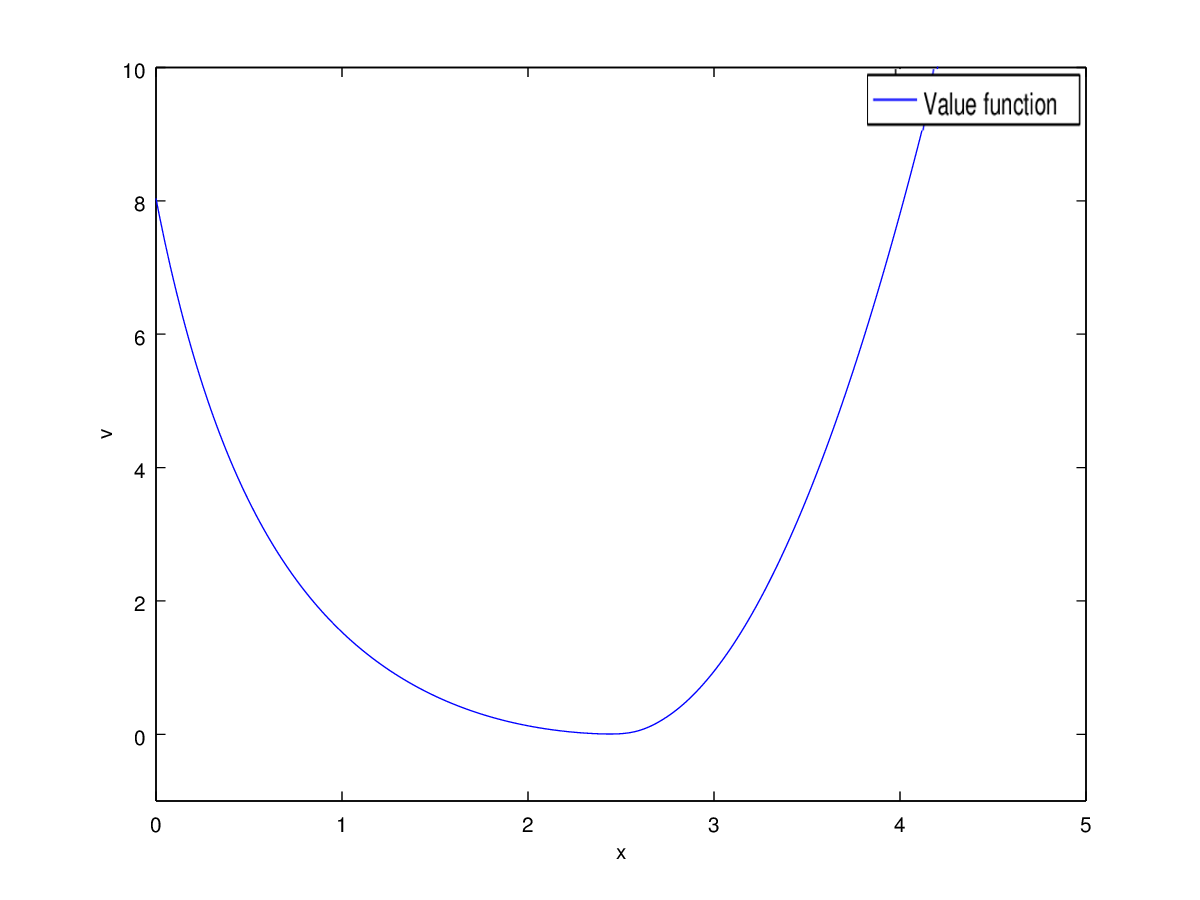} 
\includegraphics[width=0.48\columnwidth]{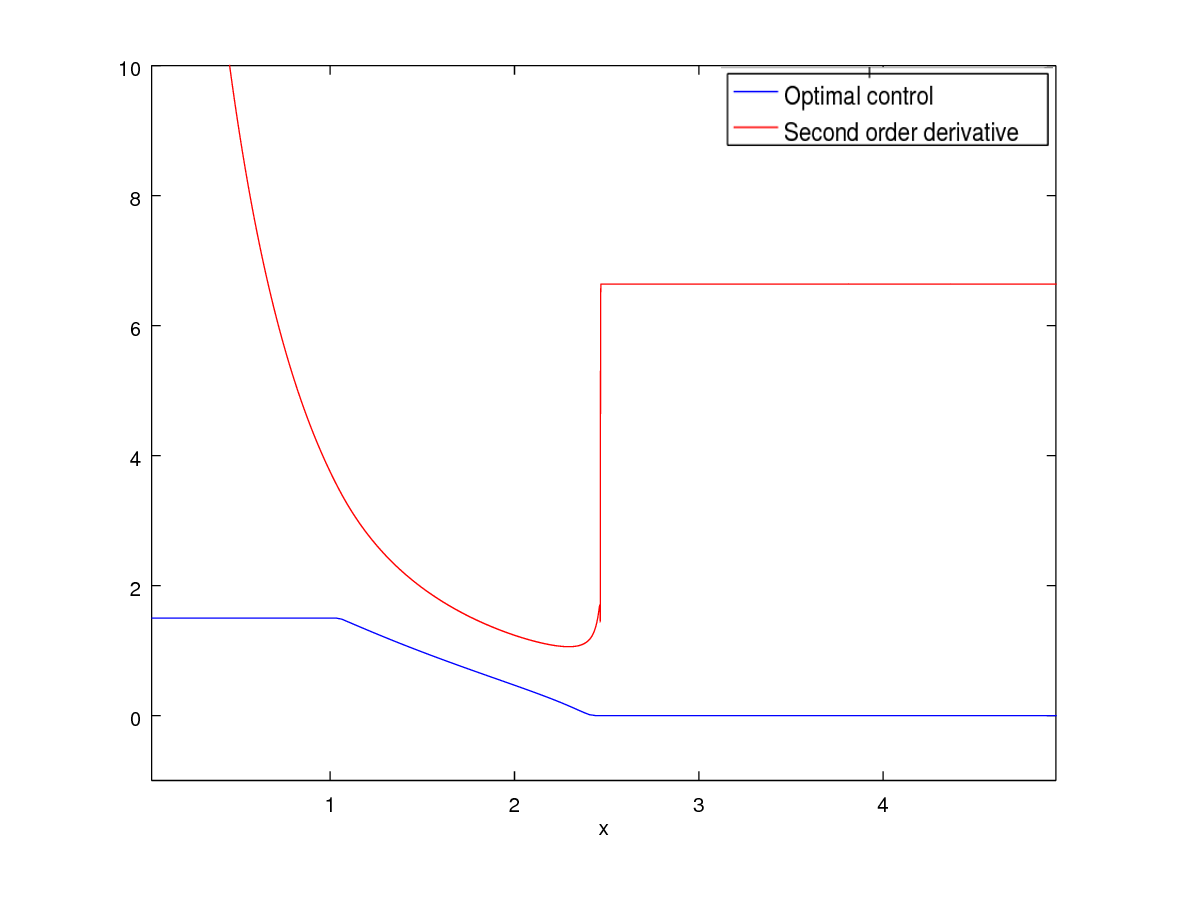}
\caption{(Example 1) Left: Plot of the value function at time $t=T$.
Right: Optimal control (blue)  at time $t=T$, and plot of $v_{xx}(T,\cdot)$ (red).
\label{fig:nonsmoothsols} 
}
\end{figure}

\begin{table}[!hbtp]
\centering
\begin{tabular}{cc|cc|cc|cc|c}
$N$ & $J$ &  Error $L^1$ & order & Error $L^2$ & order & Error $L^\infty$ & order  & CPU (s) \\
\hline   
\hline
   40 &   40  & 2.97E-01 &  1.94  & 1.64E-01 &  1.98  & 2.24E-01 &  1.85 & 0.10 \\ 
   80 &   80  & 7.64E-02 &  1.96  & 4.15E-02 &  1.98  & 5.99E-02 &  1.90 & 0.19 \\ 
  160 &  160  & 1.95E-02 &  1.97  & 1.05E-02 &  1.98  & 1.56E-02 &  1.94 & 0.34 \\ 
  320 &  320  & 4.97E-03 &  1.97  & 2.68E-03 &  1.97  & 4.01E-03 &  1.96 & 0.77 \\ 
  640 &  640  & 1.26E-03 &  1.98  & 6.86E-04 &  1.97  & 1.02E-03 &  1.98 & 1.86 \\ 
 1280 & 1280  & 3.16E-04 &  1.99  & 1.77E-04 &  1.95  & 3.35E-04 &  1.61 & 5.05 \\ 
 2560 & 2560  & 7.93E-05 &  2.00  & 4.62E-05 &  1.94  & 1.35E-04 &  1.31 & 15.43\\ 
 \end{tabular}\caption{
(Example 1) Global errors for the filtered scheme ($c_0=5$).\label{tab:bdf_high}}
\end{table}

\begin{table}[!hbtp]
\centering
\begin{tabular}{cc|cc|cc|cc|c}
$N$ & $J$ &  Error $L^1$ & order & Error $L^2$ & order & Error $L^\infty$ & order & CPU (s) \\
\hline   
\hline
   40 &   40  & 2.85E-01 &  1.96  & 1.63E-01 &  1.99  & 2.24E-01 &  1.85  & 0.10  \\ 
   80 &   80  & 7.13E-02 &  2.00  & 4.07E-02 &  2.00  & 5.99E-02 &  1.90 &  0.19 \\ 
  160 &  160  & 1.80E-02 &  1.99  & 1.02E-02 &  1.99  & 1.56E-02 &  1.94  & 0.34 \\ 
  320 &  320  & 4.52E-03 &  1.99  & 2.57E-03 &  2.00  & 4.01E-03 &  1.96  &  0.77 \\ 
  640 &  640  & 1.12E-03 &  2.01  & 6.42E-04 &  2.00  & 1.02E-03 &  1.98  & 1.86  \\ 
 1280 & 1280  & 2.79E-04 &  2.01  & 1.61E-04 &  2.00  & 2.58E-04 &  1.98 &  5.05  \\ 
 2560 & 2560  & 6.94E-05 &  2.01  & 4.03E-05 &  1.99  & 6.53E-05 &  1.98  &  15.43 \\ 
\end{tabular}\caption{(Example 1) Local errors for the filtered scheme in the subdomain $[0,5]\setminus [2.3, 2.7]$ $(c_0=5)$.
\label{tab:bdf_high10_local}}
\end{table}

\begin{figure}
\begin{minipage}{\columnwidth}
\hspace{0.7cm}
\begin{minipage}{0.9\columnwidth}
\floatbox[{\capbeside\thisfloatsetup{capbesideposition={right,center},capbesidewidth=6cm}}]{figure}
{
\caption{(Example 1) Speed of convergence for the filtered BDF2 scheme, with $\tau=4\dx$, 
and for different values of $c_0$ (errors estimated at $x=1$)
\label{fig:bdf_eps}
}}
{\hspace{-2.3 cm} \includegraphics[width=1.5\columnwidth]{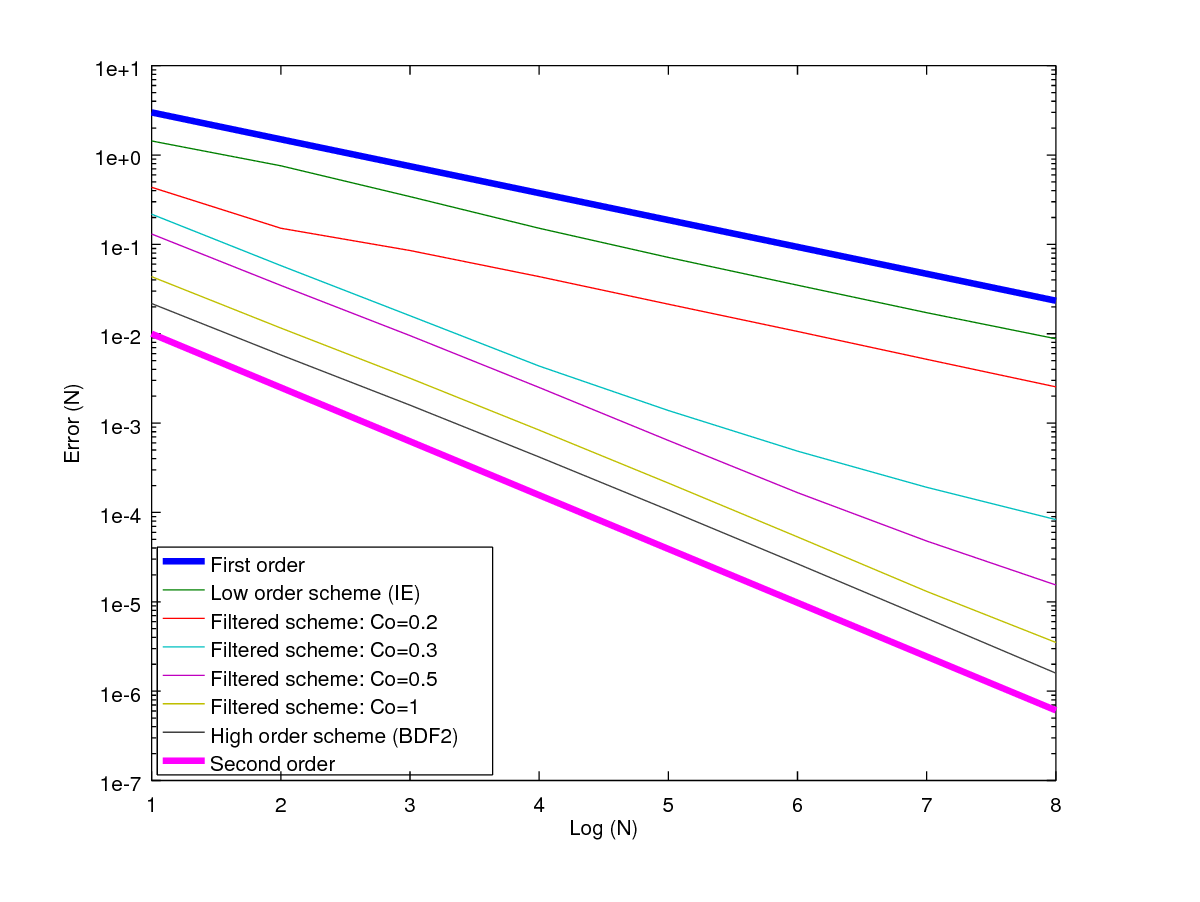}}
\end{minipage}
\end{minipage}
\end{figure}


Figure \ref{fig:bdf_eps} shows the speed of convergence of the scheme depending on the values of~$c_0$
(in this figure the errors are computed only at the point $x=1$). 
The thick blue and red lines correspond to order one and two, respectively.
For small values of~$c_0$, 
after some refinements we observe convergence of order one. 
Applying the analysis of Remark~\ref{rem:choice_eps} to this case,
in order to obtain high order we should choose $c_0$ to be such that, roughly,
\be 
  c_0 > C^v_M,
\ee
where $C^v_M$ is the constant that appears in the truncation error \eqref{eq:truncerror} of the monotone scheme.

For the implicit Euler (the monotone) scheme, the following bound is easily obtained, assuming $v$ is sufficiently regular:
\beno
  |\cE^v_{\ccS_M}(\tau,\Delta x)|\leq 
     \fud \|v_{tt}\|_\infty\tau
   + \fud \|(c + x (r + a_{\max} \sigma \xi)) v_{xx}\|_\infty \dx
   + \hbox{$\frac{1}{24}$} \|\sigma^2 a_{max}^2 x^2 v_{4x} \|_\infty \dx^2.
\eeno
Hence, neglecting the $\dx^2$ term, we obtain a bound in the form of \eqref{eq:truncerror}, i.e., 
  $|\cE^v_{\ccS_M}|\leq C^v_M \tau$ for $\tau = 4 \dx$, 
 with the constant 
$$
  C^v_M = \fud \|v_{tt}\|_\infty + \fud \|(c + x (r + a_{\max} \sigma \xi)) v_{xx}\|_\infty.
$$
Approximating the derivatives on the right-hand side 
by using the numerical solution, 
one finds the rough upper bound $C^v_M \leq 40$.
This means that we should choose $c_0\geq 40$ (and $\vare=c_0\max(\tau,\dx)$ for the filter)
in order to ensure the local error to be of high order.
However for all reported results 
we have numerically observed that there was
practically no difference for $5 \leq c_0 \leq 40$.

\subsection{Example 2: Butterfly option pricing with uncertain volatility}\label{sec:UV}

We consider the price of a financial option under a Black-Scholes model with uncertain volatility \cite{lyons1995uncertain},
with the worst scenario for a ``butterfly spread''
with respect to all the possible values of the  volatility $\sigma$ in the interval $[\sigma_{\min},\sigma_{\max}]$.   
This problem is associated with the following PDE in $\R_+ \times (0,T)$ for some expiry $T>0$:
\be
  v_t+\sup_{\sigma\in\{\sigma_{\min},\sigma_{\max}\}}\Big(-\frac{\sigma^2}{2} x^2 v_{xx} \Big)-r x v_x+r v & = & 0, \\ 
  v(0,x) & = &  v_0(x),
\ee
where 
$$
  v_0(x) :=  \max(x-K_1,0)-2\max(x-(K_1+K_2)/2,0)+\max(x-K_2,0).
$$

We follow the example presented in \cite{pooleyetal},
with parameters given by Table~\ref{tab:paramUV}, on the computational domain $[0,200]$,
using the same non-uniform space grid,
shown in Figure~\ref{fig:non-uniform}. 
The corresponding number of grid points is $J=60\times2^k$, for $k=0,1,2,\ldots$.
The time steps are uniform, given by $N=25\times 2^k$ and $\tau=\frac{T}{N}$.


\begin{figure}[!hbtp]
\footnotesize
\begin{tikzpicture}[scale=0.74]
\vspace{1cm}
\hspace{-0.2cm}
\tkzDefPoint(0,0){A}
\tkzDefPoint(4,0){B}
\tkzDefPoint(8,0){C}
\tkzDefPoint(8.8,0){D}
\tkzDefPoint(9.8,0){E}
\tkzDefPoint(10,0){F}
\tkzDefPoint(10.2,0){G}
\tkzDefPoint(11.2,0){H}
\tkzDefPoint(12,0){I}
\tkzDefPoint(16,0){L}
\tkzDefPoint(20,0){M}
\node[] at ([yshift=10pt]A){0};
\node[] at ([yshift=10pt]B){40};
\node[] at ([yshift=10pt]C){80};
\node[] at ([yshift=10pt]D){88};
\node[] at ([yshift=10pt,xshift=-2pt]E){98};
\node[] at ([yshift=10pt,xshift=2pt]G){102};
\node[] at ([yshift=10pt]H){112};
\node[] at ([yshift=10pt]I){120};
\node[] at ([yshift=10pt]L){160};
\node[] at ([yshift=10pt]M){200};
\draw[-,>=latex,line width=0.5mm] (A)--(M);
\draw[-,>=latex,line width=0.5mm] (0,0.1)--(0,-0.1);
\draw[-,>=latex,line width=0.5mm] (4,0.1)--(4,-0.1);
\draw[-,>=latex,line width=0.5mm] (8,0.1)--(8,-0.1);
\draw[-,>=latex,line width=0.5mm] (8.8,0.1)--(8.8,-0.1);
\draw[-,>=latex,line width=0.5mm] (9.8,0.1)--(9.8,-0.1);
\draw[-,>=latex,line width=0.5mm] (10.2,0.1)--(10.2,-0.1);
\draw[-,>=latex,line width=0.5mm] (11.2,0.1)--(11.2,-0.1);
\draw[-,>=latex,line width=0.5mm] (12,0.1)--(12,-0.1);
\draw[-,>=latex,line width=0.5mm] (16,0.1)--(16,-0.1);
\draw[-,>=latex,line width=0.5mm] (20,0.1)--(20,-0.1);
\draw[<->] ([yshift=-15pt,xshift=5pt]A)--([yshift=-15pt,xshift=-5pt]B);
\draw[<->] ([yshift=-15pt,xshift=5pt]B)--([yshift=-15pt,xshift=-5pt]C);
\draw[<->] ([yshift=-15pt,xshift=1pt]C)--([yshift=-15pt,xshift=-1pt]D);
\draw[<->] ([yshift=-15pt,xshift=2pt]D)--([yshift=-15pt,xshift=-8pt]F);
\draw[<->] ([yshift=-15pt,xshift=-6pt]F)--([yshift=-15pt,xshift=0pt]G);
\draw[<->] ([yshift=-15pt,xshift=2pt]G)--([yshift=-15pt,xshift=-2pt]H);
\draw[<->] ([yshift=-15pt,xshift=0pt]H)--([yshift=-15pt,xshift=0pt]I);
\draw[<->] ([yshift=-15pt,xshift=5pt]I)--([yshift=-15pt,xshift=-5pt]L);
\draw[<->] ([yshift=-15pt,xshift=5pt]L)--([yshift=-15pt,xshift=-5pt]M);
       
\tkzDefPoint(2,0){A1}
\tkzDefPoint(6,0){B1}
\tkzDefPoint(8.4,0){D1}
\tkzDefPoint(9.3,0){E1}
\tkzDefPoint(10,0){F1}
\tkzDefPoint(10.7,0){G1}
\tkzDefPoint(11.6,0){H1}
\tkzDefPoint(14,0){I1}
\tkzDefPoint(18,0){L1}
\node[] at ([yshift=-28pt]A1){\scriptsize$\frac{10}{2^k}$};
\node[] at ([yshift=-28pt]B1){\scriptsize$\frac{5}{2^k}$};
\node[] at ([yshift=-28pt]I1){\scriptsize$\frac{5}{2^k}$};
\node[] at ([yshift=-28pt]D1){\scriptsize$\frac{2}{2^k}$};
\node[] at ([yshift=-28pt]H1){\scriptsize$\frac{2}{2^k}$};
\node[] at ([yshift=-28pt]E1){\scriptsize$\frac{1}{2^k}$};
\node[] at ([yshift=-28pt]G1){\scriptsize$\frac{1}{2^k}$};        
\node[] at ([yshift=-28pt]F1){\scriptsize$\frac{0.5}{2^k}$};        
\node[] at ([yshift=-28pt]L1){\scriptsize$\frac{10}{2^k}$};
\node[] at ([yshift=-26pt,xshift=0pt]A){\footnotesize$\Delta x$};
\vspace{1cm}
\end{tikzpicture}
\normalsize
\caption{\label{fig:non-uniform} (Example 2) Non-uniform mesh steps for $k=0,1,2,\dots$.}
\end{figure}
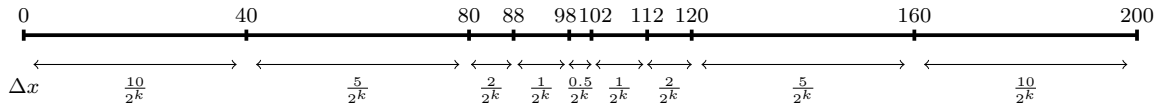


\begin{table}[!hbtp]
\centering
\begin{tabular}{|c|c|c|c|c|c|}
\hline 
$r$ & $\sigma_{\min}$ & $\sigma_{\max}$ &  $T$ & $K_1$ & $K_2$\\
\hline
$0.1$ & $0.15$ & $0.25$ & $0.1$  & $90$ & $110$ \\
\hline
\end{tabular}
\caption{Parameters used in numerical experiments for uncertain volatility  problem.}
\label{tab:paramUV}
\end{table}

Figure~\ref{fig:CNvsFilter} 
shows the result obtained using the CN scheme (left) as in \cite{pooleyetal},
the corresponding result obtained adding a filter to the CN scheme (center), and the BDF2 scheme (right).
The monotone scheme used for filtering is the standard implicit Euler finite difference scheme \eqref{eq:IE}. 

In \cite{pooleyetal}, it is shown that, in absence of a CFL condition constraining $\tau$ 
to be of the order of $\Delta x^2$, the Crank-Nicolson (CN) 
scheme may not converge to the unique viscosity solution of the problem. 
This is also illustrated in Figure~\ref{fig:CNvsFilter} (left).
We refer to \cite{pooleyetal} 
for the details of the CN scheme used.

Numerical results are given in Table~\ref{tab:UV_comparison},
for the CN scheme,
the ``CN-Rannacher'' scheme (see below),
the BDF2 scheme~\eqref{eq:BDF2},
the filtered CN scheme (with $\vare:=50\dx_{\min}$)
and the filtered BDF2 scheme~\eqref{eq:BDF2} (with $\vare:=50 \dx_{\min}$). 
The table reports the $L^\infty$-norm of the error using as reference an accurate numerical solution computed with the BDF2 scheme with $N=12800$ and $J=30720$. Here the CN scheme is not convergent to the correct value. However, 
by considering the error estimated as differences of two successive values obtained at $x=100$,
and the corresponding ``order'' of convergence (see Table~\ref{tab:UV_comparison}), 
one may observe a slow convergence
-- towards a wrong value, as already explained in~\cite{pooleyetal}.

The CN-Rannacher scheme corresponds to a CN scheme with Rannacher time-stepping, i.e., 
the fully implicit Euler scheme is used for computing $u^1$ and $u^2$ (first two steps $n=0,1$), then it switches to the CN scheme 
(for computing $u^{n+1}$ for steps $n\geq 2$), as in \cite{pooleyetal}.

On this example, the filter is able to correct the wrong behavior of the CN scheme, but
the overall order of convergence is not greater than one.
This is due to the fact that the filter is applied in a quite wide area, as seen in Figure \ref{fig:CNvsFilter},
where the dots on the $x$-axis in the middle panel indicate the mesh points where the filter is active.
In particular, the measure of the area where the filter 
is applied does not diminish as the mesh is refined and therefore the contribution to the overall error 
from the low order scheme dominates.

Table \ref{tab:UV_comparison} also compare the orders of convergence. 
Both the CN-Rannacher scheme and the BDF2 scheme appear to converge to the correct viscosity solution with second order.
However, for these last two schemes, because of the lack of monotonicity, convergence is not theoretically established.

In order to ensure convergence to the viscosity solution, we can apply the filter to the BDF2 scheme. 
The convergence for different values of $c_0$ is shown in Figure \ref{fig:UVeps}.
However, for this example, 
the unboundedness of the derivatives for $t \downarrow 0$ and the resulting unboundedness 
of the constant $C^v_M$ do not allow high order of convergence in the relevant region of the domain.

\begin{figure}
\centering{$t=0$}\\
\includegraphics[width=0.3\columnwidth]{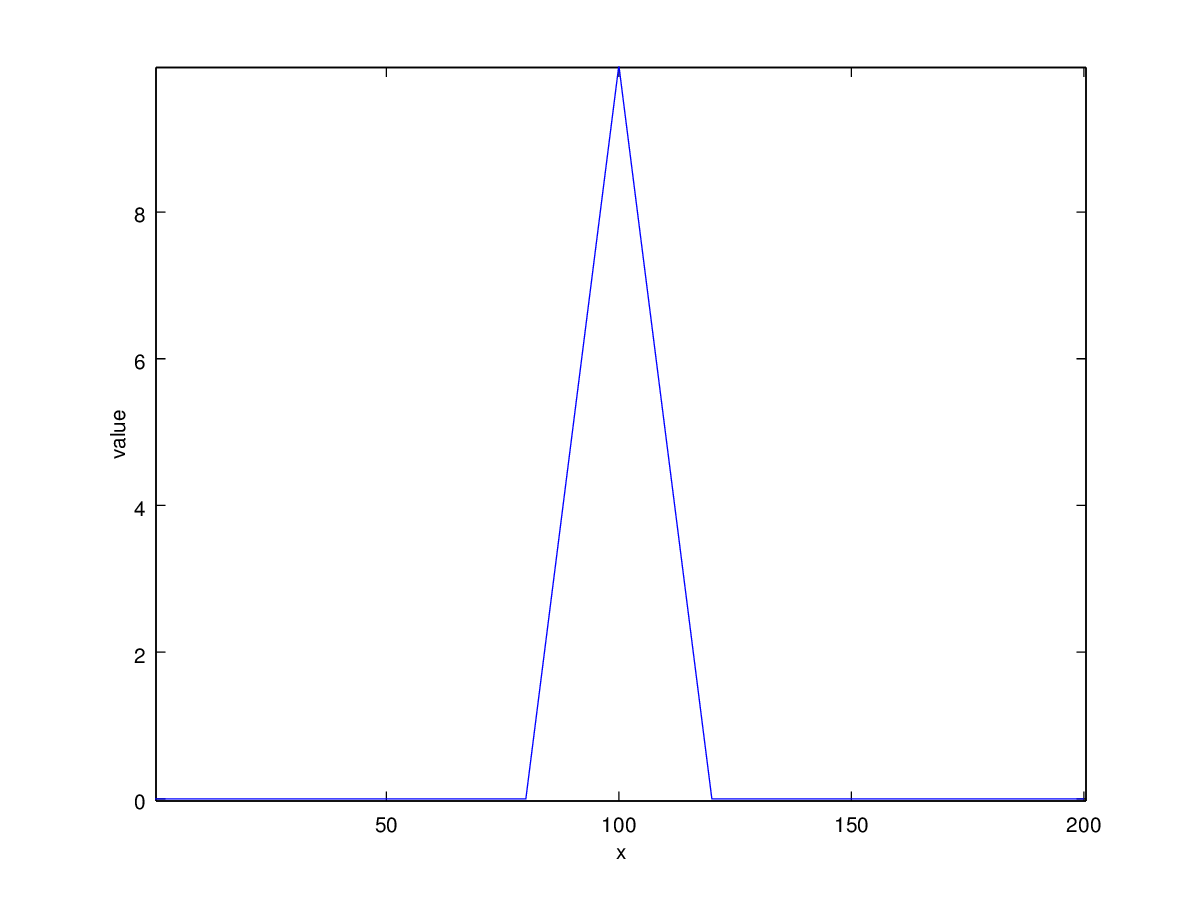}
\includegraphics[width=0.3\columnwidth]{payoff_butt.png}
\includegraphics[width=0.3\columnwidth]{payoff_butt.png}\\
\centering{$t=\tau$}\\
\includegraphics[width=0.3\columnwidth]{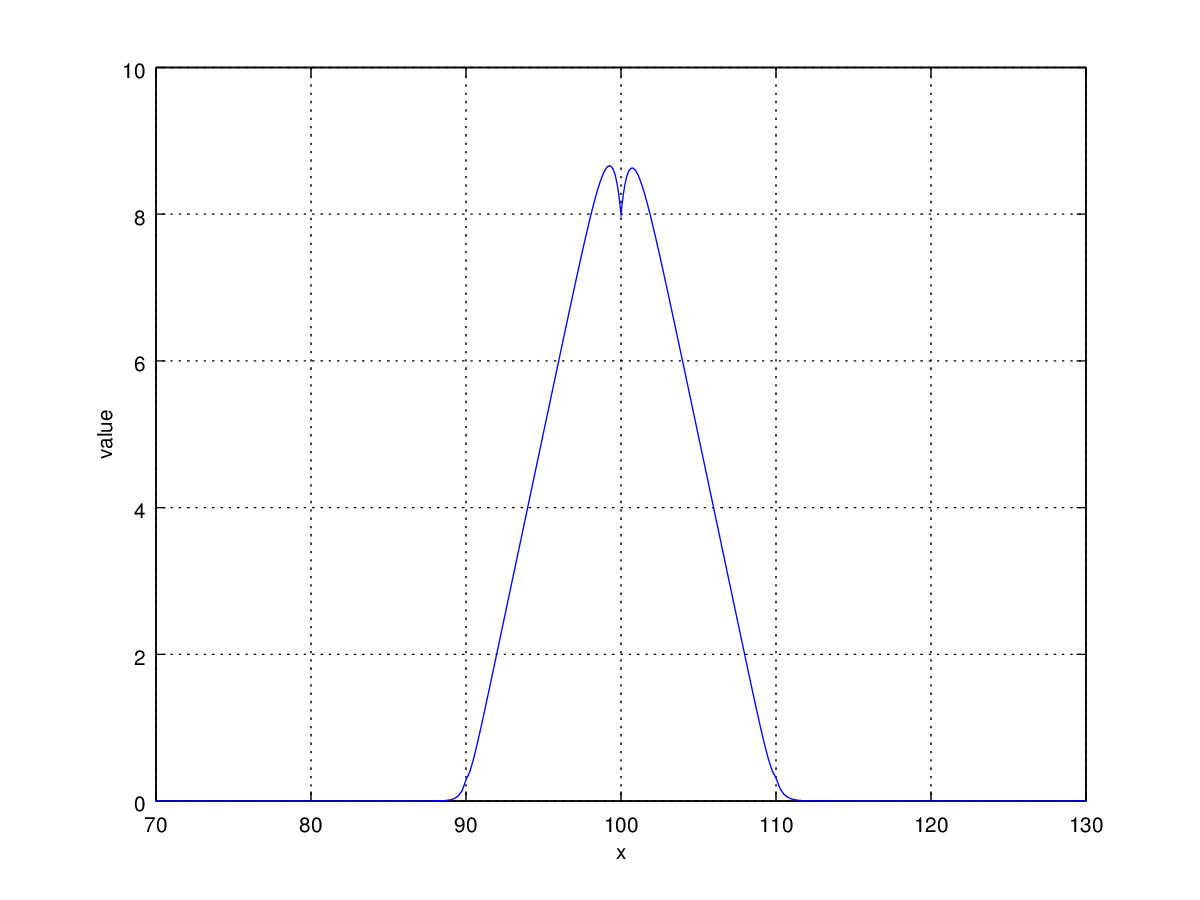}
\includegraphics[width=0.3\columnwidth]{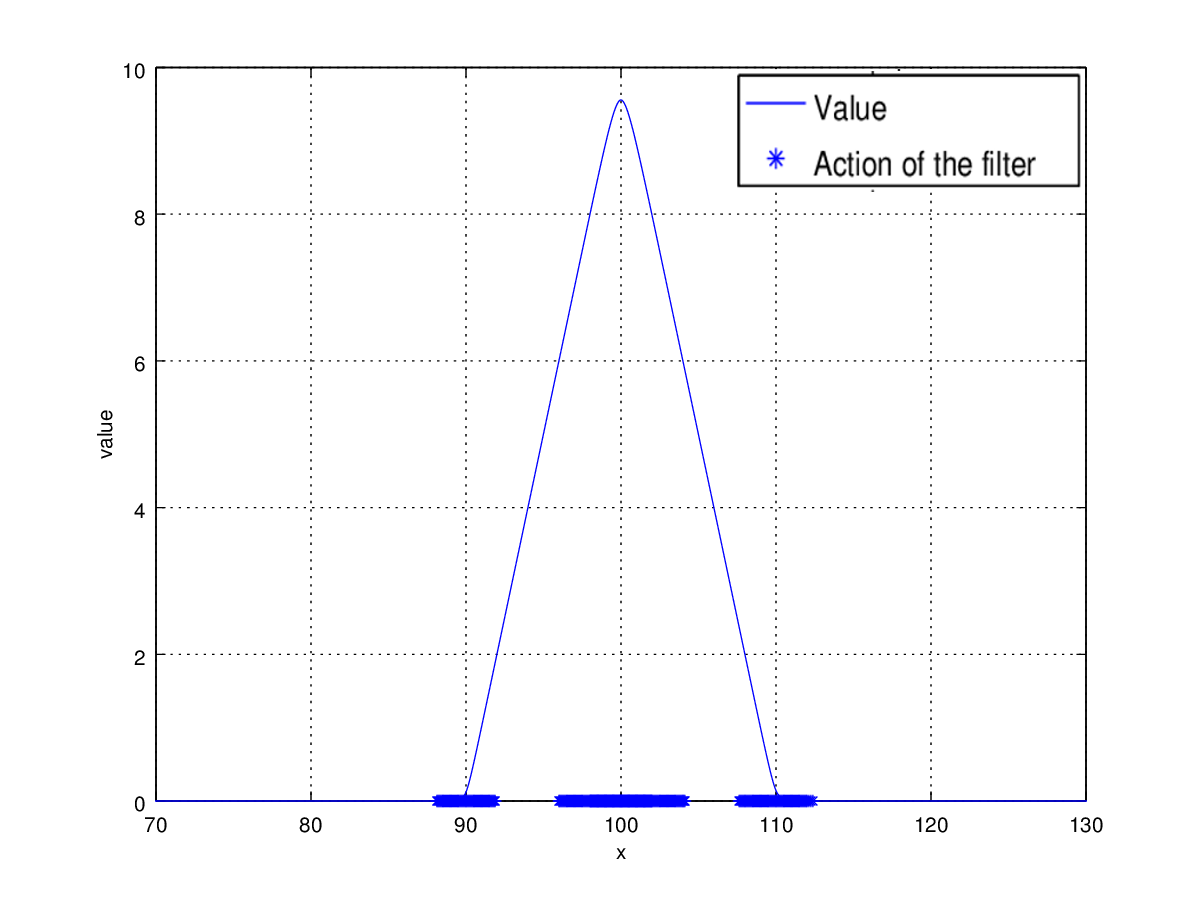}
\includegraphics[width=0.3\columnwidth]{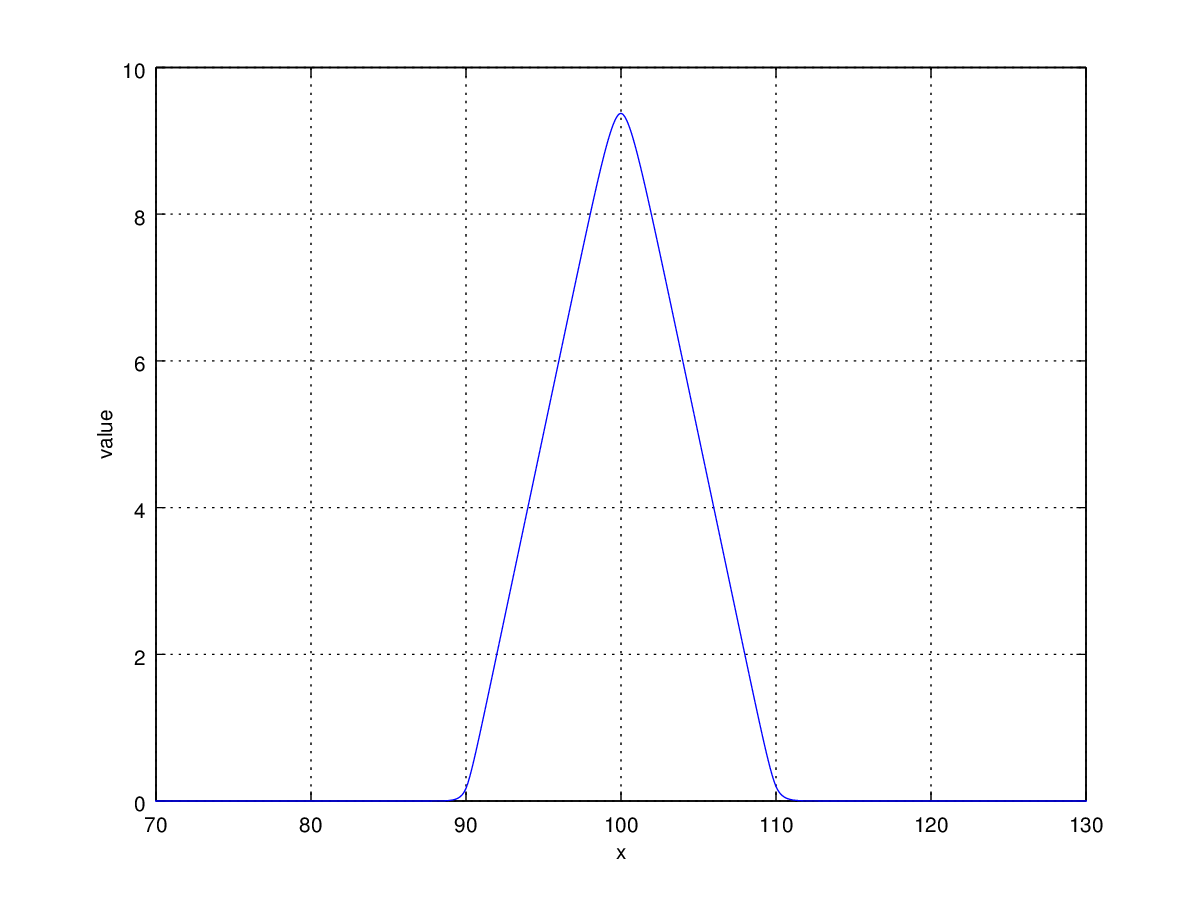}\\
{\centering $t=T$}\\
\includegraphics[width=0.3\columnwidth]{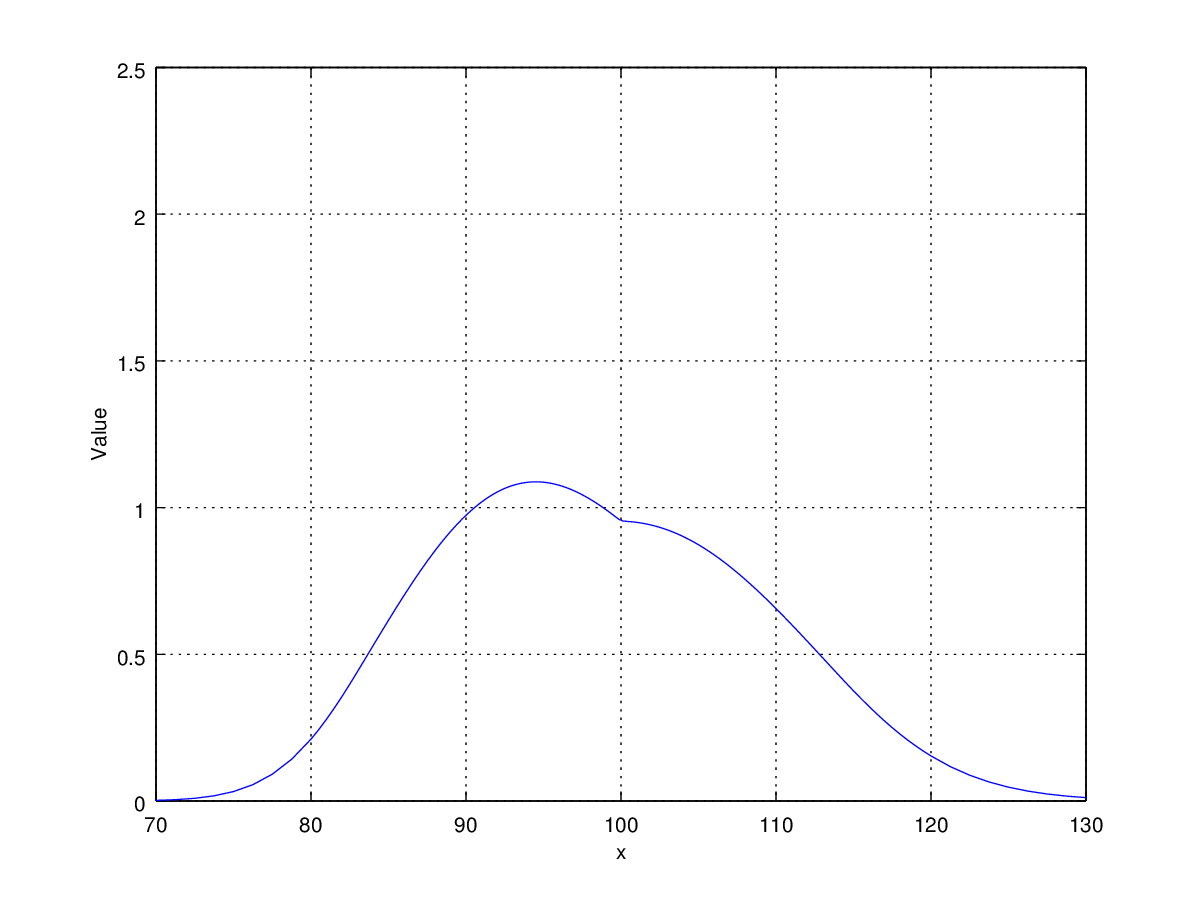}
\includegraphics[width=0.3\columnwidth]{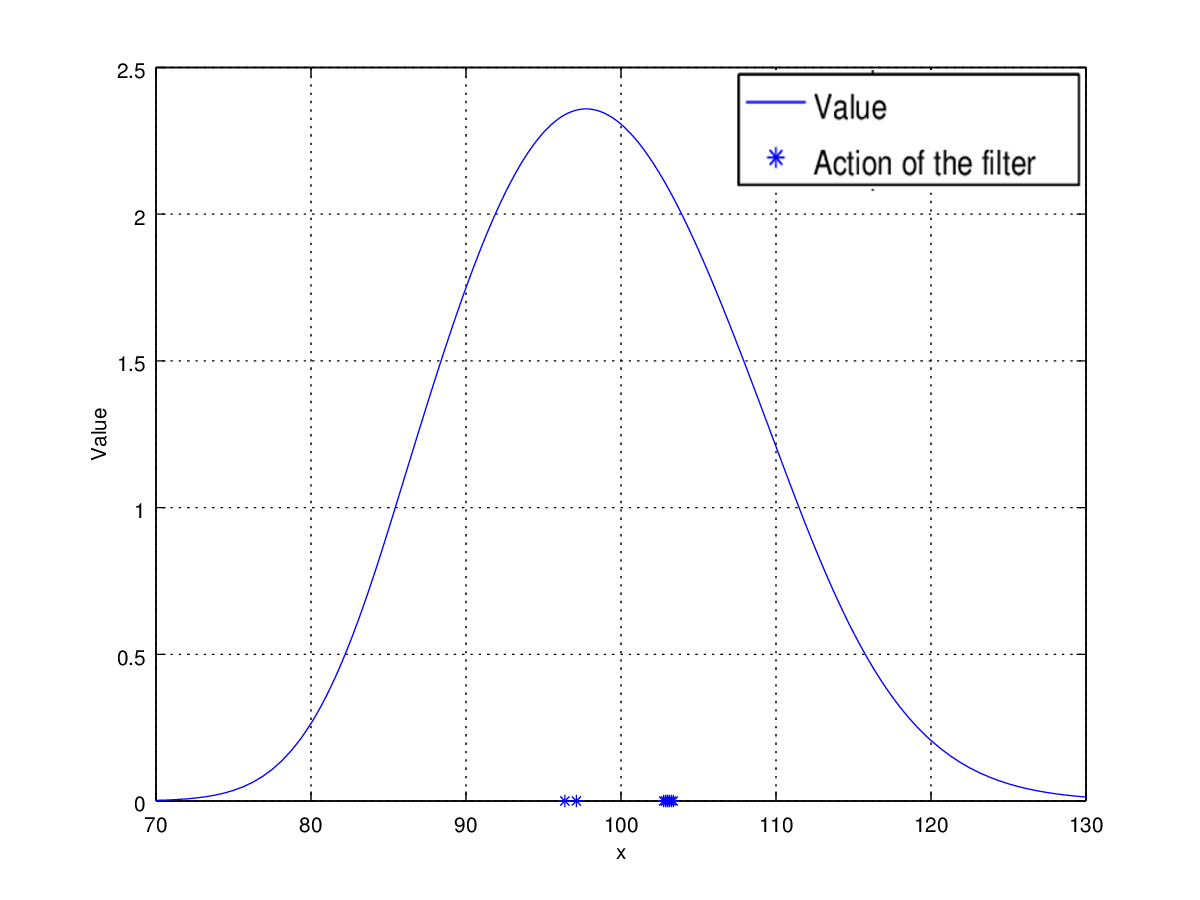}
\includegraphics[width=0.3\columnwidth]{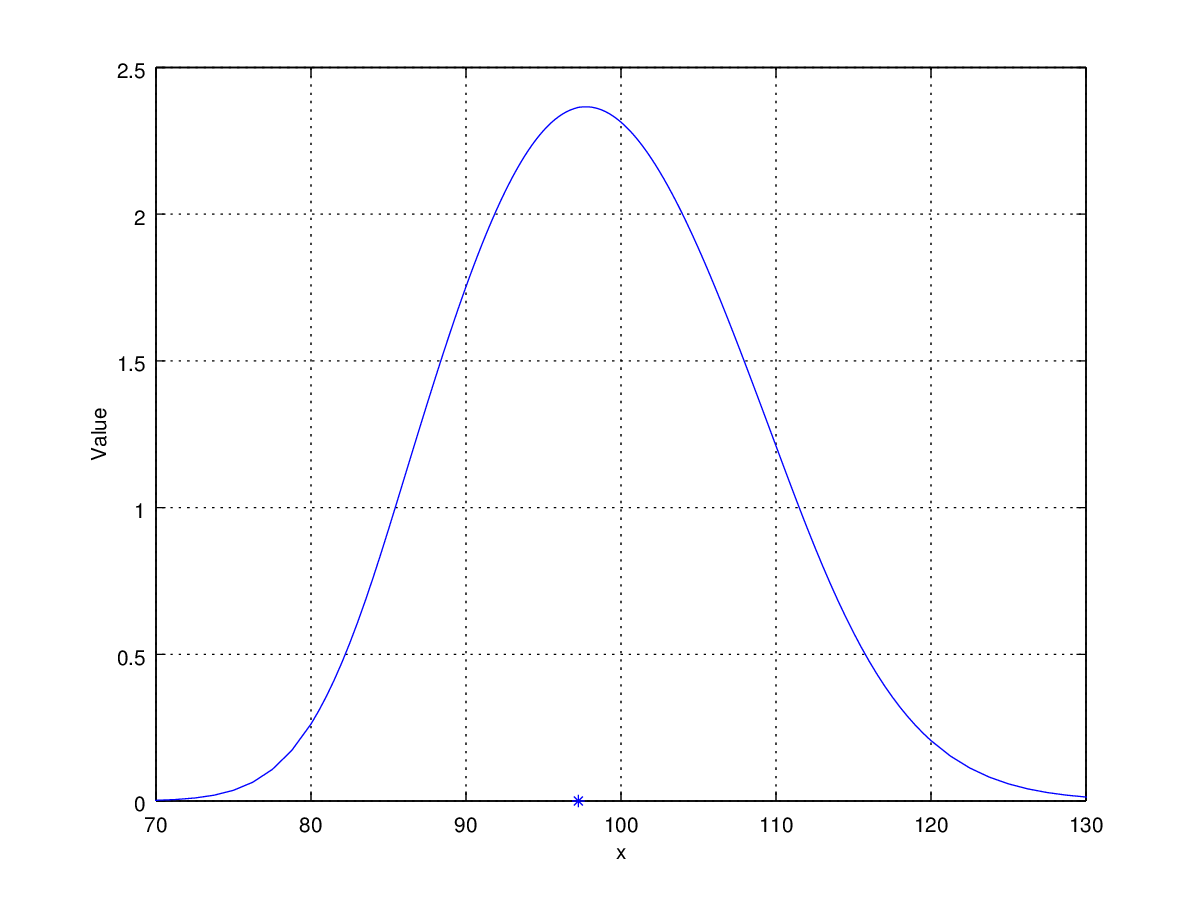}
\caption{(Example 2) From top to bottom: value function at $t=0$ (the payoff), after one time-step $t=\tau$, and at terminal time $t=T$.
From left to right: CN scheme, CN+filter scheme, BDF scheme.
 \label{fig:CNvsFilter}}
\end{figure}

\begin{small}
\begin{table}[!hbtp]
\centering
\footnotesize
\begin{tabular}{c|cc|cc|cc|cc|cc}
  $N$ 
   & \multicolumn{2}{|c|}{CN}
   & \multicolumn{2}{|c|}{CN-Rannacher}  
   & \multicolumn{2}{|c|}{BDF2} 
   & \multicolumn{2}{|c|}{CN+filter} 
   & \multicolumn{2}{|c }{BDF2+filter} \\
\hline\hline
       \if{&   --}\fi  & value & ``order''   & error    & order & error    & order & error     & order &  error   &  order \\
    25 \if{&   60}\fi  & 1.884 & -           & 3.38E-02 & -     & 3.19E-02 & -     & 5.58E-01  & -     & 3.19E-02 & -      \\
    50 \if{&  120}\fi  & 1.060 & -           & 9.51E-03 & 1.82  & 9.53E-03 & 1.74  & 3.34E-01  & 0.74  & 9.50E-03 &   1.75 \\
   100 \if{&  240}\fi  & 0.957 & 0.32        & 2.38E-03 & 1.99  & 2.58E-03 & 1.88  & 1.10E-01  & 1.60  & 2.88E-03 &   1.72 \\
   200 \if{&  480}\fi  & 0.884 & 0.50        & 5.94E-04 & 2.00  & 6.71E-04 & 1.94  & 5.10E-02  & 1.11  & 1.07E-03 &   1.43 \\
   400 \if{&  960}\fi  & 0.835 & 0.55        & 1.48E-04 & 2.00  & 1.71E-04 & 1.97  & 2.01E-02  & 1.34  & 3.79E-04 &   1.50 \\
   800 \if{& 1920}\fi  & 0.802 & 0.58        & 3.69E-05 & 2.00  & 4.30E-05 & 1.99  & 1.87E-02  & 0.10  & 1.97E-04 &   0.95 \\
  1600 \if{& 3840}\fi  & 0.780 & 0.62        & 9.11E-06 & 2.01  & 1.07E-05 & 2.00  & 1.48E-02  & 0.34  & 9.84E-05 &   1.00 \\
  3200 \if{& 7680}\fi  & 0.767 & 0.64        & 2.15E-06 & 2.08  & 2.55E-06 & 2.06  & 8.74E-03  & 0.76  & 5.04E-05 &   0.97
\end{tabular}
\caption{\label{tab:UV_comparison}
(Example 2) Orders of convergence for the CN scheme (convergent towards a wrong solution),
the CN scheme with Rannacher time-stepping,
the BDF2 scheme,
the CN scheme with filter, 
and the BDF2 scheme with filter ($\vare = 50\dx_{\min}$).
Here using $N=25\times 2^k$ and $J=60\times 2^k$, $k=0,1,2...$.
}
\normalsize
\end{table}
\end{small}


\if{

\begin{small}
\begin{table}[!hbtp]
\centering
\begin{footnotesize}
\begin{tabular}{cc|cc|cc}
  $N$ & $J$ 
  & \multicolumn{2}{|c}{IE}
  & \multicolumn{2}{|c}{BDF2 + filter}
  \\ 
  \hline \hline 
       &      &   error &  order  & error       & order \\  
%
    25 &   60 & 3.18E-02  &   -    & 5.53E-02  & -         \\
    50 &  120 & 2.90E-02  & 0.93   & 9.53E-03  &   1.74    \\
   100 &  240 & 1.48E-02  & 0.96   & 2.58E-03  &   1.88    \\
   200 &  480 & 7.53E-03  & 0.98   & 6.70E-04  &   1.94    \\
   400 &  960 & 3.79E-03  & 0.99   & 5.92E-04  &   0.18    \\
   800 & 1920 & 1.90E-03  & 0.99   & 3.49E-04  &   0.76    \\
  1600 & 3840 & 9.52E-04  & 1.00   & 8.31E-05  &   2.07    \\
  3200 & 7680 & 4.76E-04  & 1.00   & 4.02E-05  &   1.05    
\end{tabular}
\end{footnotesize}
\caption{
\label{tab:filteredBDF_VS_IE}
(Example 2)
\MODIF{ $L^\infty$-error and order of convergence for the IE scheme and the filtered BDF2 scheme ($\vare=500\Delta x_{\min}$).
Both schemes have first order of convergence,
but we can observe that the absolute error is almost by a factor of 10 lower than for the pure low order scheme.} 
\QUESTION{PUT ALL FILTER SCHEMES HERE ? - waiting for tables for $c_0=10$ for both CN and BDF2}
}
\end{table}
\end{small}

}\fi

\begin{figure}
\begin{minipage}{\columnwidth}
\hspace{0.7cm}
\begin{minipage}{0.9\columnwidth}
\floatbox[{\capbeside\thisfloatsetup{capbesideposition={right,center},capbesidewidth=6cm}}]{figure}
{\caption{(Example 2) Convergence rate of the $L^\infty$-error obtained for $\vare = c_0 \Delta x_{\min}$ 
and different values of $c_0$, using the non uniform mesh defined in \cite{pooleyetal}. The error is computed by comparison with an accurate numerical solution obtained for $N=12800$ and $J=30720$. \label{fig:UVeps}} }
{\hspace{-2.3 cm} \includegraphics[width=1.5\columnwidth]{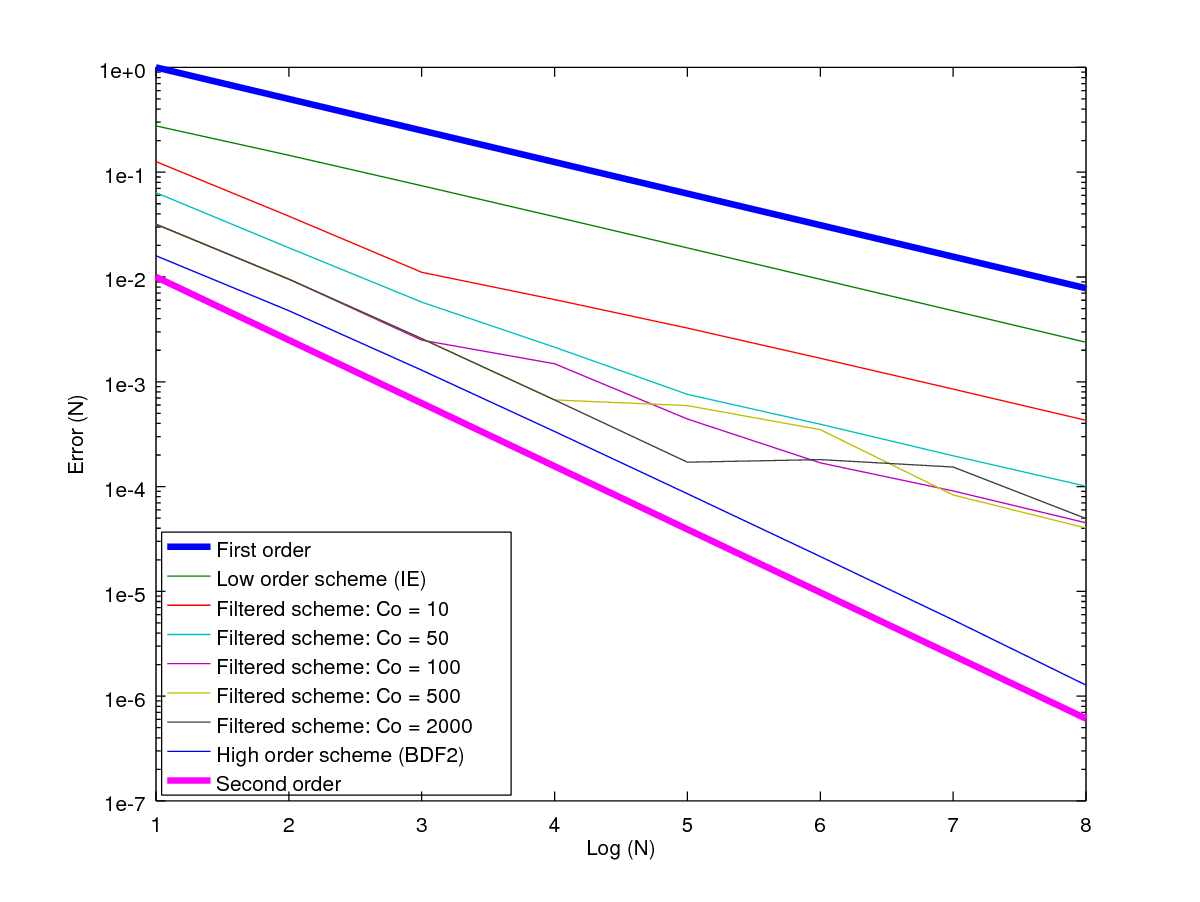}}
\end{minipage}
\end{minipage}
\end{figure}

\subsection{Example 3: A two-dimensional example}\label{sec:2d}
 
We now test a filtered scheme on a two-dimensional example set on $\Omega=(-\pi,\pi)^2$:
\begin{subequations}
\be
  v_t + \sup_{a \in A} \Big(-\frac{1}{2}Tr\big(\ms_a\ms_a^T D^2_x v\big) 
   - \ell(t,x,a) \Big)\ =\ 0, &&\quad t\in(0,T), x\in \Omega\label{eq:2d1} 
  \\
 v(0,x)\ =\ 2\sin x_1\sin x_2, && \quad x\in\Omega\label{eq:2d2}
\ee
\end{subequations}
with periodic boundary conditions, $T:=0.5$,
$$
  A:= \{a=(a_1,a_2)\in \R^2: a_1^2+a_2^2 = 1\},
$$
$$
  \ms_a := \sqrt{2}\left(\begin{array}{c} a_1 \\ a_2\end{array}\right), 
  \qquad \ell(t,x,a) := (1-t)\sin x_1\sin x_2 + (2-t)(a_1^2 \cos^2 x_1+a_2^2 \cos^2 x_2 ).
$$ 
This problem has the following exact solution 
$$
  v(t,x) = (2-t) \sin x_1 \sin x_2.
$$
This example is similar -- but more complex as concerns the solution of the optimization problem -- to the one discussed in \cite[Section 9.3 (B)]{deb-jak-12}. 
In our case, the minimization over $a$ is non-trivial, in contrast to \cite{deb-jak-12}, where the problem reduces to a linear one and the solution satisfies the PDE for any control -- not only for the optimal one.

Because of the  presence of cross-derivatives and the fact that the diffusion matrix 
$\ms_a \ms_a^T$ 
may not be diagonally dominant, standard  finite difference schemes (such as the seven-point stencil) are generally not monotone. 
Therefore, for the monotone scheme we consider here a semi-Lagrangian scheme (as in \cite{deb-jak-12}).

Let $N\geq 1$ and let $\tau = T/N$ be the time step. We use a uniform space mesh $x_{ij}=~(x_{1,i},x_{2,j})$ in the two dimensions with mesh steps 
$\dx_1,\dx_2$ such that
$$
  \Delta x_1 = \Delta x_2 = \frac{2\pi}{J}.
$$

The monotone scheme gives an approximation $u^{n}_{ij}$ of $v(t_n,x_{1,i},x_{2,j})$ as in \eqref{eq:SL1}, i.e.
\beno
  \frac{u^{n+1}_{ij} - u^n_{ij}}{\tau} = \underset{a\in {A}}\inf
  \bigg(\frac{[u^n](x_{ij} + \sqrt{\tau}\ms_a)-2u^n_{ij} +[u^n](x_{ij} - \sqrt{\tau}\ms_a)}{2 \tau}+\ell(t_n,x_{ij},a) \bigg)
\eeno
or, equivalently (as in \eqref{eq:SL2})
\beno
   u^{n+1}_{ij} =  \underset{a\in {A}} \inf\bigg( \harf \sum_{\eps=\pm 1} [u^n](x_{ij} + \eps \sqrt{\tau}\ms_a) 
   +  \tau \ell(t_n,x_{ij},a) \bigg)
\eeno
where $[\ \cdot \ ]$ stands for the bilinear -- $Q_1$ -- interpolation,
and $u^0_{ij}=v(0,x_{ij})$. \\


The infimum is approximated using the following discretization of the controls, replacing ${A}$ by
\be \label{eq:discrTheta}
{A}_P = \{a_k,\  0\leq k\leq P-1\}, \quad a_k = \left(\cos\left(\frac{2k\pi}{P}\right), \sin\left(\frac{2k\pi}{P}\right)\right),
\ee
for some $P\in \N$, $P\geq 1$, and we denote also $\da:=\frac{2\pi}{P}$ a control mesh step.

\if{
\begin{rem}\label{rem:linearity2d}
In the case of \eqref{eq:2d1}-\eqref{eq:2d2}, the solution is linear in time, and since the optimal control depends on $x$ but not on $t$, 
this behaviour is inherited by the numerical scheme.
Therefore the following formula holds: 
\begin{equation*}
\frac{u^{N} - u^0 }{T}=  \underset{a\in {A}}\inf\Big\{\frac{[u^0](x + k\sigma(a))-2u^0
 +[u^0](x - k\sigma(a))}{{2k^2}}+\ell(t_0,x,a)\Big\},
\end{equation*}
which allows to compute directly $u^N$ in this particular case.
\end{rem}
}\fi

The high order scheme we consider here is an implicit finite difference scheme
based on the following naive approximation of the second order derivatives for $\phi^n_{ij} \equiv \phi(t_n,x_{1,i},x_{2,j})$:
\beno
  \partial^2_{xx} \phi_{ij}:=\frac{\phi_{i+1,j} - 2 \phi_{i,j} + \phi_{i-1,j}}{\Delta x^2},\qquad 
  \partial^2_{yy} \phi_{ij}:=\frac{\phi_{i,j+1} - 2 \phi_{i,j} + \phi_{i,j-1}}{\Delta y^2},
\eeno
\beno
  \partial^2_{xy} \phi_{ij}:=\frac{\phi_{i+1,j+1} -  \phi_{i+1,j-1} + \phi_{i-1,j-1} - \phi_{i-1,j+1}}{4\Delta x\Delta y}.
\eeno
Hence, denoting
$\partial^2 u:= \begin{pmatrix} \partial^2_{xx} u & \partial^2_{xy} u \\
                                \partial^2_{xy} u & \partial^2_{yy} u \end{pmatrix}$,
the scheme is
\be \label{eq:2d-FD}
  \frac{u^{n+1}_{ij}- u^n_{ij}}{\tau} + \sup_{a\in A_P} \Big(- \fud Tr\big(\ms_a \ms_a^T\, \partial^2 u^{n+1}\big)_{ij}
    + \ell(t_{n+1},x_{ij},a) \Big)  = 0.
\ee

The problem with the monotonicity of the high-order scheme here
does not come from the timestepping scheme, but only from 
the finite difference approximation of the spatial derivatives.

\if{
\MODIF{
As pointed out in Remark \ref{rem:linearity2d}, it also holds, for $u^N$, the following formula 
\be\label{eq:2dFD}
\frac{u^N - u^0}{T}=\underset{a\in {A}}\inf\ \Big\{a_1^2 D^2_{xx} u^N +2 a_1 a_2 D^2_{xy} u^N +a_2^2  D^2_{yy} u^N +\ell(t_N,x,a)  \Big\}.
\ee
}
}\fi

Equation \eqref{eq:2d-FD} is solved  by policy iteration \cite{bok-mar-zid-09}. 
Even though we have no proof of convergence of the policy iteration
algorithm in this setting, we have numerically observed fast convergence. 

In order to find a suitable value of $c_0$ for the filtered scheme (see Remark~\ref{rem:choice_eps}), the constant $C_{M}$ 
that appears in the truncation error of the monotone scheme
needs to be estimated.
For the SL scheme above, a consistency estimate similar to \eqref{eq:EM} holds with 
$$
  \big|\cE^v_{\ccS_M}(\tau,\dx,\da)\big|\leq 
   \frac{\tau}{2}\|v_{tt}\|_\infty 
  + 
   \frac{2}{3} \tau \| D^4 v\|_\infty 
  +\frac{1}{8}\frac{(\dx_1^2 +\dx_2^2)}{\tau} \| D^2 v\|_\infty + \frac{\sqrt{10} \|D^2v\|_\infty + 2}{4} \da^2,
$$
where we have denoted 
$\|D^p v\|_\infty:=\max_{k=0,\dots,p} \|\frac{\partial^4 v}{\partial x^k y^{p-k}}\|_\infty$.
We take $\da$ of the same order as $\dx$ and $\tau$ (i.e. $\da=\frac{2\pi}{P}=\dx_1=\dx_2=\frac{2\pi}{J}$,
and, with $T=0.5$, $\tau=\frac{1}{2N}\equiv \frac{\dx_1}{4\pi}$),
so that the error coming from $\da$ becomes negligible.
\COMMENTED{
\footnote{
\QUESTION{LOOK HERE FOR DETAILS - to be removed for final version} 
For $A\in \R^{2\times 2}$, Let $\|A\|_F = Tr(A A^T)^{1/2}$ and 
$\|D^2 v\|_{F,\infty}=\max_x \|D^2v(x)\|_F$. In particular, 
$\|D^2 v\|_{F,\infty}^2 \leq \|v_{xx}\|^2_\infty +  \|v_{yy}\|^2_\infty + 2 \|v_{xy}\|^2_\infty$.
Let $H:=(\dx_1,\dx_2)$. 
Elementary calculus gives $|v_{i+1,j+1}-2 v_{ij} + v_{i-1,j-1}|\leq \|D^2 v\|_{F,\infty} \|H\|^2_2$.
Then for $\dx_1=\dx_2$,
$|\partial^2 v_{xy}|\leq 2 \frac{(\dx_1^2+\dx_2^2)}{4\dx_1 \dx_2} \|D^2 v\|_{F,\infty} \leq \|D^2 v\|_{F,\infty}$. 
Also 
$|\partial^2 v_{xx}|\leq \|v_{xx}\|_\infty$, 
$|\partial^2 v_{yy}|\leq \|v_{yy}\|_\infty$.
So  
$\|\partial^2 v\|_{F,\infty}^2 
  \leq \|v_{xx}\|_\infty^2 + \|v_{yy}\|_\infty^2 + 2 \|D^2 v\|_{F,\infty}^2 
  \leq 10 \|D^2 v\|_{\infty}^2$,
So $\|\partial^2v\|_F\leq \cred{\sqrt{10}} \|D^2 v\|_{\infty}$.
Let $f(a):=Tr(a a^T \partial^2 v)+ \ell(t,x,a)$.
If $a$ is an extrema of $f(.)$, then 
$|f(a+h)-f(a)|
   \leq  \| D^2 f\|_{F,\infty}  \| h h^T \|_F 
   \leq  (\|\partial^2 v\|_{F,\infty}+2)\|h\|_2^2
   \leq  (\sqrt{10}\|D^2 v\|_{\infty}+2)\|h\|^2_2$.
Furthermore we see that $\|h\|_2 \leq \frac{1}{2}\|a_k-a_{k+1}\|_2 \leq \frac{1}{2}\frac{2\pi}{P} = \frac{1}{2} \da$.
Hence the extrema in $a$ is reached up to an error bounded by $ \frac{\sqrt{10} \|D^2v\|_\infty + 2}{4} \da^2$.
}
}
Using that $v_{tt}=0$ and $\|D^2v\|_\infty=\|D^4 v\|_\infty\leq 2$
in this example, this gives the bound
$ \big|\cE^v_{\ccS_M}(\tau,\dx,\da)\big|\leq  C^v_M\, \tau$ with a constant $C^v_M$ such that: 
$$
  C^v_M \leq \frac{4}{3} + 4\pi^2 \simeq 40.
$$
In Figure \ref{fig:eps2dLin} different convergence orders are observed, using $\vare = c_0 \tau$ with
 different values of $c_0$. 
We observe convergence of second order already for $c_0=0.8$, 
which is consistent with the upper bound 40
(see Remark~\ref{rem:choice_eps}).

Table~\ref{tab:2dLin} shows the results for $c_0=0.8$.
As the table shows, we obtain second order convergence for all norms and refinement levels considered.
The computational complexity is $O(N J^2 P)$, as is confirmed in the last column of Table~\ref{tab:2dLin}.

Figure~\ref{fig:ActiveFlin}, left, shows the solution, and the right plot the points of activity of the filter for different values of~$c_0$.
As soon as $c_0\geq 0.8$ we do not observe any use of the filter.


\begin{figure}
\centering
\includegraphics[width=0.7\columnwidth]{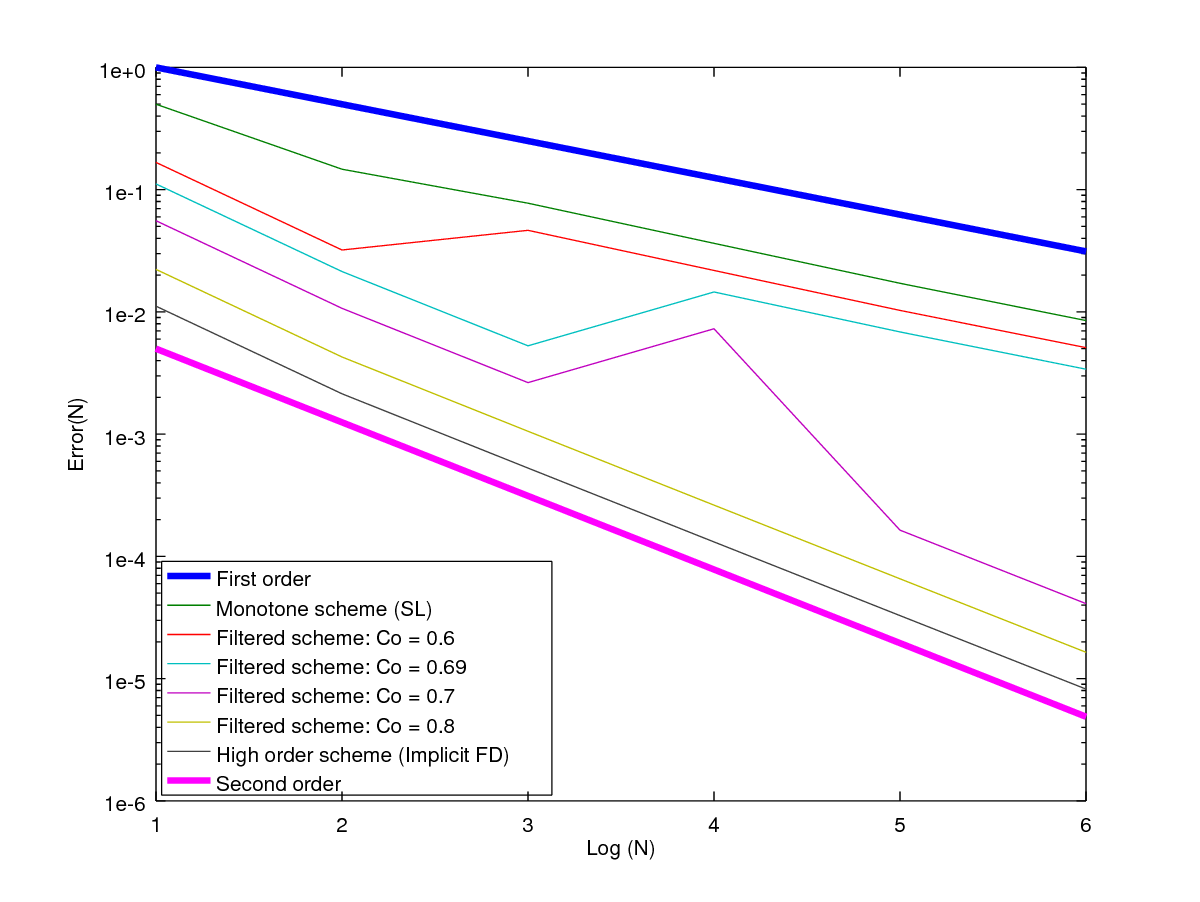}
\caption{(Example 3) Order of convergence in the $L^\infty$-norm for different values of $c_0$. 
Second order convergence is observed for $c_0\simeq0.8$ or greater values of $c_0$.
\label{fig:eps2dLin}                 
}
\end{figure}


\begin{figure}
\centering
\includegraphics[width=0.5\columnwidth]{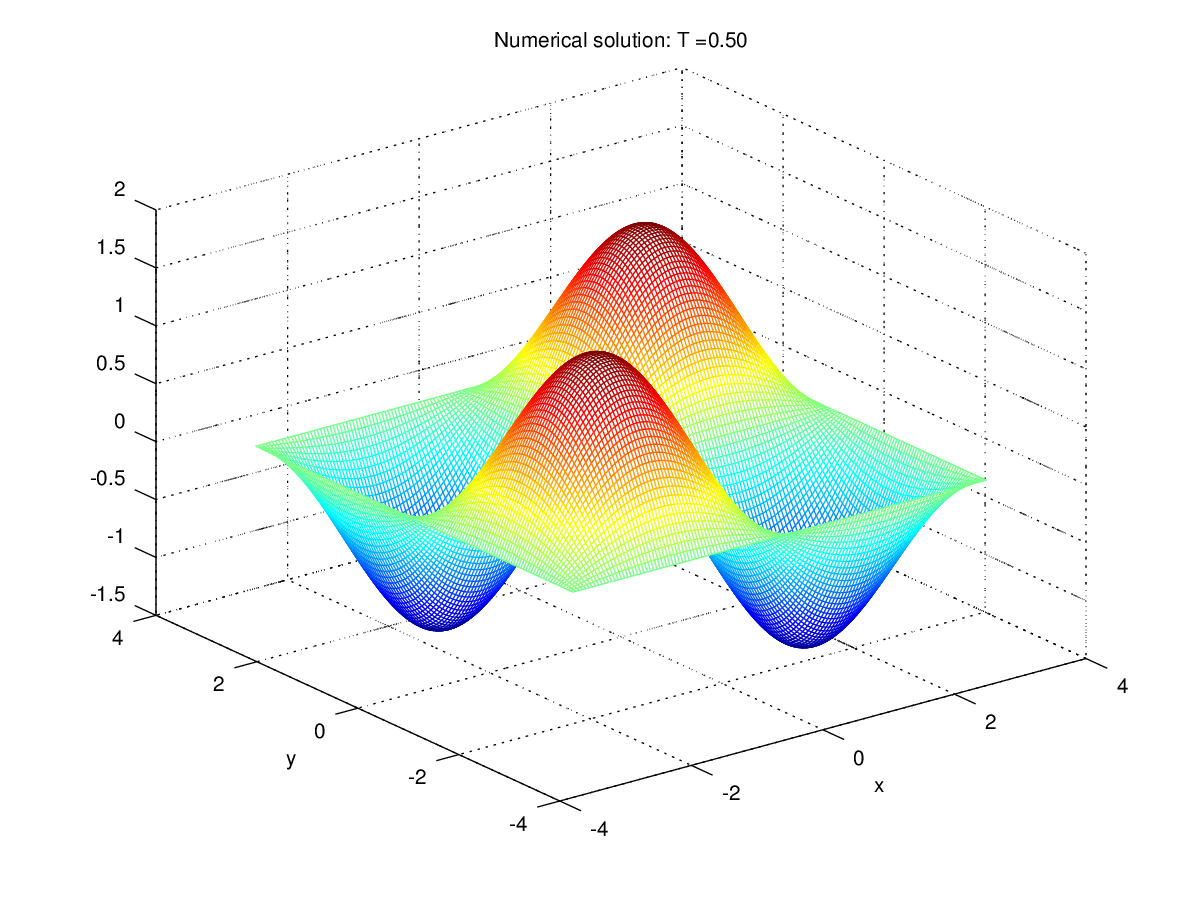}
\hspace{-1 cm}
\includegraphics[width=0.5\columnwidth]{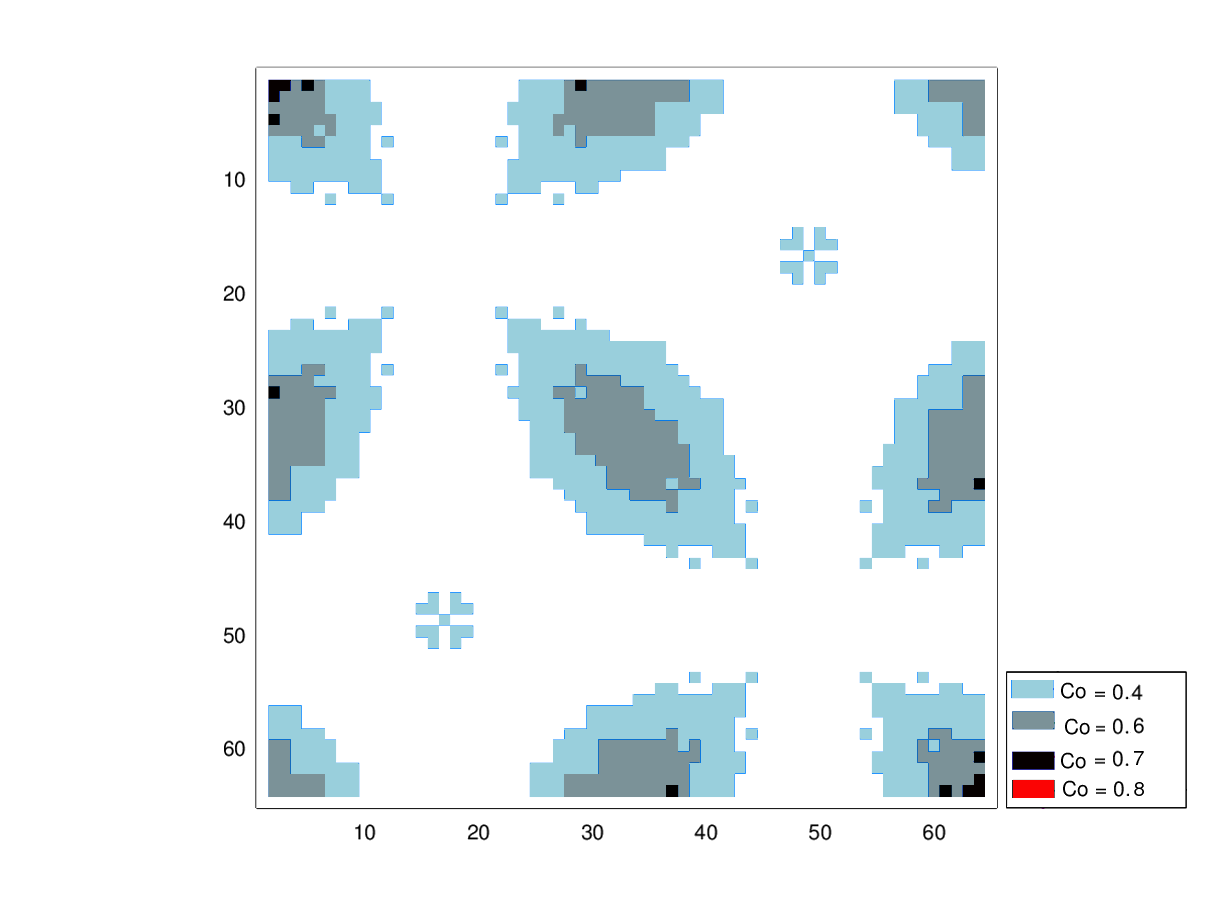}
\caption{(Example 3)
Left: value function at time $T=0.5$. Right: points of activity of the filter corresponding to different values of $c_0$. 
For $c_0\geq 0.4$ we do not observe any use of the filter.}\label{fig:ActiveFlin}
\end{figure}\begin{table}[!hbtp]
\centering
\begin{tabular}{ccc|cc|cc|cc|r}
$N$ & $J$ & $P$  & Error $L^1$ & order  & Error $L^2$ & order & Error $L^\infty$ & order & CPU (s)\\
\hline\hline
{   4} &     4  &   4 & 1.13E+00 &   -    & 9.27E-01 &   -    & 5.57E-01 &   -   &      1.3\\ 
{   8} &     8  &   8 & 4.16E-01 &  1.44  & 2.79E-01 &  1.73  & 1.08E-01 &  2.37 &      5.8\\ 
{  16} &    16  &  16 & 1.39E-01 &  1.58  & 9.61E-02 &  1.54  & 3.13E-02 &  1.78 &     71.9\\ 
{  32} &    32  &  32 & 3.76E-02 &  1.89  & 2.70E-02 &  1.83  & 8.82E-03 &  1.83 &   1068.2\\ 
{  64} &    64  &  64 & 9.58E-03 &  1.97  & 7.06E-03 &  1.94  & 2.31E-03 &  1.93 &  19380.0\\  
{ 128} &   128  & 128 & 2.51E-03 &  1.93  & 1.83E-03 &  1.95  & 6.10E-04 &  1.92 & 307450.0
\end{tabular}
\caption{(Example 3) Error and order of convergence for the filtered scheme with $c_0=0.8$.
\label{tab:2dLin}}
\end{table}

\section{Conclusions}\label{sec:concl}

Filtered schemes are designed to combine the advantages of the guaranteed convergence of low order monotone schemes 
and the superior accuracy -- in regions where the solution is smooth -- of higher order non-monotone schemes. 
The theoretical results in this paper confirm these properties.

In our numerical tests, the schemes delivered the accuracy of the high order scheme if the solution is smooth. 
For an example with a locally non-smooth solution, the filter was seen to turn a divergent higher order time stepping scheme (the Crank-Nicolson scheme) 
into a convergent scheme, albeit only at the order of the low-order scheme. 
Although non-monotone high order schemes with better stability (such as the BDF2 scheme) empirically gave second order convergence,
the filter reduced this order to one due to singularities of the higher order derivatives of the solution resulting in a wide application of the filter.


Ongoing works concern a more intrinsic choice of the $\vare$ parameter that is used in the filtered scheme.


%

\end{document}